\documentclass[12pt]{article}
\usepackage{amssymb,amsmath, amsthm}
\usepackage{graphicx}

\newtheorem*{theo}{Theorem}
\newtheorem{thm}{Theorem}

\newtheorem*{propo}{Proposition}

\newtheorem{ques}[thm]{Question}

\def\demo{\medskip\goodbreak\noindent
     \hbox{\sc Proof \kern .3em}\ignorespaces}%
  \def \qedbox{$\square$}%
  \def \qed{\hglue1mm\hfill{\ifmmode\qedbox
     \else\unskip\ \hglue0mm\hfill\qedbox\medskip
      \goodbreak\fi}}%

\def\qed{\hglue1mm\hfill\raise -2pt\hbox{\vrule\vbox to 10pt{\hrule width
4pt
                  \vfill\hrule}\vrule}}

\newcommand{\T}{\mathbb {T}}

\newcommand{\R}{\mathbb {R}}

\newcommand{\Z}{\mathbb {Z}}
\newcommand{\N}{\mathbb {N}}
\newcommand{\Nc}{\mathcal {N}}

\newcommand{\Hc}{\mathcal {H}}

\newcommand{\Mc}{\mathcal {M}}

\newcommand{\Gc}{\mathcal {G}}
\newcommand{\Lc}{\mathcal {L}}
\newcommand{\LLL}{\widehat L}
\newcommand{\LL}{\widetilde L}
\newcommand{\fit}{\wtilde \varphi}

\newcommand{\Fc}{\mathcal {F}}
\newcommand{\Ac}{\mathcal {A}}

\newcommand{\Tc}{\mathcal {T}}

\newcommand{\cku}{{C^k (\wtilde{U})}}
\newcommand{\ckpu}{{C^{k+1} (\wtilde{U})}}

\newcommand{\ep}{\epsilon}
\newcommand{\dps}{\displaystyle}

\newtheorem{theorem}{Theorem}
\newtheorem{lemma}{Lemma}[theorem]

\newtheorem{proposition}{Proposition}
\newtheorem{corollary}{Corollary}[theorem]
\newtheorem{remark}{Remark} 
\newtheorem{definition}{Definition}
\newcommand{\be}{\begin{equation}}
\newcommand{\ee}{\end{equation}}

\newcommand{\om}{\omega}

\newcommand{\ov}{\overline}
\newcommand{\wtilde}{\widetilde}
\newcommand{\G }{\Gamma }

\newcommand{\la}{\langle}
\newcommand{\ra}{\rangle}

\newcommand{\TT}{T^*\T^n}

\def \a{\alpha}

\newcommand{\e}{\varepsilon}
\def \g{\gamma}
\def \G{\Gamma}
\def \l{\lambda}
\def \o{\omega}
\def \pr{{\rm pr}}
\def \vv{\vert \vert}

\def \ec{extremal curve }
\def \ecs{extremal curves }
\def \cp{conjugate points }

\begin{document}

\title{ Tonelli Hamiltonians without conjugate points and $C^0$ integrability}
\author{M.~Arcostanzo\thanks{Avignon University, LMA EA 2151,  F-84000, Avignon, France}, M.-C.~Arnaud\footnotemark[1] \thanks{ANR-12-BLAN-WKBHJ} \thanks{membre de l'Institut Universitaire de France} , P.~Bolle\footnotemark[1], M.~Zavidovique\footnotemark[2] \thanks{IMJ, Universit\' e Pierre et Marie Curie, Case 247,
4 place Jussieu, 
F-75252 Paris cedex 05 } }
\maketitle
\abstract{We prove that all the Tonelli Hamiltonians defined on the cotangent bundle $T^*\T^n$ of the $n$-dimensional torus that have no conjugate points are $C^0$ integrable, i.e. $T^*\T^n$ is $C^0$ foliated by a family $\Fc$ of invariant $C^0$ Lagrangian graphs.\\
Assuming that the Hamiltonian is $C^\infty$,  we prove that there exists a $G_\delta$ subset  $\Gc$ of $\Fc$ such that the dynamics restricted to every element of $\Gc$ is strictly ergodic.\\
Moreover, we prove that the Lyapunov exponents of  every $C^0$ integrable Tonelli Hamiltonian are zero and deduce that the metric and topological entropies vanish. }
\medskip

  
 \newpage

\tableofcontents

\section*{Introduction}
This article deals with $C^0$ integrable Tonelli Hamiltonians\footnote{All the notions will be defined at the end of this introduction} and Tonelli Hamiltonians without conjugate points of the cotangent bundle $T^*\T^n$ of the $n$-dimensional torus. 

If the Tonelli Hamiltonian  is a Riemannian metric, these properties coincide and  have strong implications. 
Indeed, Heber showed (see \cite {Heber}) in 1994 that for every Riemannian 
metric without conjugate points on the torus $\T^n$, there is a 
continuous foliation of the unit tangent bundle by tori which are Lipschitz, 
Lagrangian and invariant by the geodesic flow. The same year, this was 
improved by Burago and Ivanov who proved (see \cite {BI}) that such a metric 
has to be flat; as an immediate consequence, the continuous foliation in 
Heber's result is in fact smooth.

The notion of Tonelli Lagrangian is a vast extension of the concept of 
Riemannian metric, but we shall prove in the first  section that Heber's result 
still holds:


\begin{theorem}\label{th1}
Let $H$ be a Tonelli Hamiltonian on $T^*\T^n$. Then $H$ 
has no conjugate points if and only if there is a continuous foliation 
of $T^*\T^n$ by Lipschitz, Lagrangian, flow-invariant graphs.
\end{theorem}

The proof uses ideas coming from weak KAM and Aubry-Mather theory. In fact, we establish that each leaf of the above foliation is the dual Aubry set corresponding to some cohomology class. More precisely, the first step of the proof is to see that some of those Aubry sets (later denoted by $\Gc^*_{T,r}$, with $T>0$ and $r\in \mathbb Z^n$) are covered by periodic orbits of the Hamiltonian flow $(\phi_t^H)_{t\in \R}$, of a given period $T$ (and a given homology class $r$). In particular, the dynamics on the corresponding leaves is periodic.

The existence of those particular leaves is used again in the second  section. Using a KAM theorem, we prove that such sets $\Gc^*_{T,r}$ are accumulated by KAM tori on which the dynamics is conjugated to an irrational rotation:

\begin{theorem}
 \label{quasi}
Assume that $H$ is $C^\infty$ and that $\ov{\om}$ is  strongly Diophantine\footnote{This notion is defined in  subsection \ref{ss20}}. There is $m_0 \in \N\backslash \{0 \}$ (depending on $\ov{\om}$) such that, for all
$m \geqslant m_0$, there is a $C^\infty$ Lagrangian embedding $i_m : \T^n \to \TT $ such that
\begin{itemize} 
\item[i)]
$ \  \forall \eta \in \T^n \, , \quad
\phi^H_{mT} \big(i_m(\eta)\big)=i_m(\eta + \ov{\om}) \, $. 
\item[ii)]
Writing $i_m(\eta)=\big(\psi_m (\eta), f_m (\eta)\big)$, $\psi_m$ is a $C^\infty$ diffeomorphism of $\T^n$, isotopic to $id_{\T^n}$
 and
$$
{\cal T}_m= i_m (\T^n)= \big\{ \big(\theta \,  ,\,  (f_m \circ \psi_m^{-1})(\theta)\big) \, ; \ \theta \in \T^n \  \big\} \, .
$$
is a Lagrangian graph; the sequence $({\cal T}_m)$ converges to ${\cal T}_\infty:= \Gc_{T,r}$ in $C^\infty$ topology. 
\item[iii)]
The sequence $(\psi_m)$ 
converges in $C^\infty$ topology to a diffeomorphism $\psi_\infty$ of $\T^n$   (independent of $\ov{\om}$),
isotopic to $id_{\T^n}$.
\item[iv)] The tori ${\cal T}_m$ are  flow-invariant. More precisely,
$ \ i_m^* (X_H)=\dps \frac{r}{T} + \frac{\ov{\om}}{mT} \ $, so that
$$
\forall m \geqslant m_0 \, , \ \forall t\in \R \, , \ 
\forall \eta \in \T^n \, , \quad  {\phi}^H_t \big({i}_m (\eta)\big)
={i}_m \Big(\eta + \frac{t}{T} r + \frac{t}{mT}\ov{\om}\Big).
$$
\end{itemize}
\end{theorem}
A corollary of this theorem is that the dynamics of the Hamiltonian flow restricted to $\Gc^*_{T,r}$ is itself conjugated to a rational rotation on $\T^n$.
An important ingredient in the proof of this theorem is to provide a normal form for the flow of $H$ in the neighborhood of $\Gc^*_{T,r}$, linking it to the geodesic flow of a flat metric on the torus. This is done in proposition \ref{NF},   with the use of the theorem of Burago and Ivanov, and is of independent interest. 

Another consequence of this theorem is that we can deduce some information on the dynamics restricted to a lot of invariant tori. More precisely, let us recall that an invariant set is {\em strictly ergodic} if there is only one Borel invariant probability measure the support of whose   is in this set, and if the support of this measure is the whole set.

\begin{theorem}
Let $(\phi^H_t)$ be a $C^\infty$ Tonelli flow of $T^*\T^n$ with no conjugate points  and let $\mathfrak F$ be the continuous foliation in invariant Lagrangian tori that is given by theorem \ref{th1}. Then there is a dense $G_\delta$ subset $\Gc$ of $\mathfrak F$ such that, for every $\Tc \in\Gc$, then $\phi^H_{1|\Tc}$ is strictly ergodic.
\end{theorem}


The last section of this article is devoted to studying the entropy of Tonelli Hamiltonians without conjugate points.
 Indeed, it is not hard to see that a regular completely integrable Hamiltonian system has zero topological entropy. When singularities are allowed, the situation can become more complicated, as shown in the article \cite{bolsinov} of Bolsinov and Taimanov.\\
   Hence it seems  natural to ask what can happen for a $C^0$ integrable Tonelli Hamiltonian. In this case, we don't  know the restricted dynamics to all the invariant tori, hence it is not so obvious to decide if the topological entropy is zero or not. An answer to this question is provided by the following:
    \begin{theorem}\label{t4}
   Let $H:T^*\T^n\rightarrow \R$ be a $C^3$ Tonelli Hamiltonian that is $C^0$ integrable. Then for every invariant Borel probability measure, the Lyapunov exponents are zero.
    \end{theorem}
    This implies that both the metric, and topological entropies must also be $0$. \\
    Observe that the conclusion of theorem \ref{t4} is true for a $C^0$ integrable Tonelli Hamiltonian defined on $T^*M$ for any closed manifold $M$. We give the statement for $T^*\T^n$ because we define Tonelli Hamiltonians just in this case, but the interested reader can find a definition of Tonelli Hamiltonians on any manifold in \cite{Fathi}.
    
  Some interesting questions remain open after this work, as:
  \begin{ques}
  Does a $C^0$ integrable Tonelli Hamiltonian exist that is not $C^1$ integrable?
  \end{ques}
  \begin{ques}
  Can an invariant torus of a $C^0$ integrable Tonelli Hamiltonian flow carry two invariant measures that have not the same rotation number (see the appendix for the definition of the rotation number)?
  \end{ques}

\subsection*{Notations and definitions}
\subsubsection*{Tonelli Lagrangians and Hamiltonians}
Let $\T^n = \R^n / \Z^n$ denote the $n-$dimensional torus, with  
$$\pr : \R^n \longrightarrow \T^n, \ \pi : (x,v) \in T\T^n \longmapsto 
x \in \T^n, \ {\rm or} \ \pi : (x,p) \in T^*\T^n \longmapsto x \in \T^n $$
\noindent 
the canonical projections (or their lifts to $T\R^n$ or $T^*\R^n$ when no confusion is possible). For every $(x,v) \in T\T^n$, the vertical space  
at this point is 
$$V(x,v) = {\rm{Ker}} (D\pi_{\vert T_{(x,v)}T\T^n}),$$
 and for every $(x,p) \in T^*\T^n$, the vertical space  
at this point is $V^*(x,p) = {\rm{Ker}} (D\pi_{\vert T_{(x,p)}T^*\T^n})$.

A function $L:T\T^n \to \R$ is named a Tonelli Lagrangian if it verifies the following three conditions:
\begin{itemize}
\item it is of regularity at least $C^2$,
\item it is super-linear: $\lim\limits_{|v|\to \infty} \frac{L(x,v)}{|v|} =+ \infty$,
\item it is strictly convex in the fibers in the sense that $\partial^2_{v}L$ is everywhere positive definite as a quadratic form.
\end{itemize}
Let $L$ be a ($C^k$, $k \geqslant2$) Tonelli Lagrangian on $T\T^n$, and $\LL$ its ($\Z^n\times\{ 0\}$-periodic) lift 
to $T\R^n$. Its associated Hamiltonian is defined by 
$$\forall (x,p)\in T^*\T^n,\quad H(x,p)=\sup_{v\in T_x\T^n}\big( p\cdot v-L(x,v)\big).$$
It is also Tonelli and its lift to $T^* \R^n$ is $\wtilde H$, the Hamiltonian associated to $\wtilde L$. The Lagrangian and Hamiltonian are linked by the Legendre transform:
$$\Lc : T\T^n \to T^*\T^n,\quad (x,v) \mapsto \big(x, \partial_v L(x,v)\big),$$ as follows:
$$\forall (x,v)\in T\T^n,\quad H\circ \Lc (x,v) + L(x,v) =  \partial_v L(x,v) \cdot v.$$
Note that the Legendre transform is a global $C^{k-1}$ diffeomorphism, but $H$ is of class $C^k$ as $L$. 

The Hamiltonian flow is generated by the vector-field on $T^*\T^n$: 
$$(x,p)\mapsto X_H(x,p)=(\partial_p H,-\partial_x H).$$
 The flow it generates is denoted $(\phi_t^H)_{t\in \R}$ and it is complete and $C^{k-1}$.
 We shall denote by $\varphi_t : T\T^n \longrightarrow T\T^n$ 
the Euler-Lagrange flow of $L$, which is conjugated to $\phi_t^H$ by $\Lc$. Similarly $\fit_t$ (resp. $\wtilde \phi_t^H$) will be the flow of $\LL$ (resp. $\wtilde H$). Recall that $H$ (or equivalently $L$)  is {\sl without conjugate points} 
if one has 
$$ \forall (x,v) \in T\T^n, \ \forall t \in \R^*, \quad
V\big(\varphi_t(x,v)\big) \cap D\varphi_t(x,v)\cdot V(x,v) = \{ 0 \},$$
or equivalently
$$ \forall (x,p) \in T^*\T^n, \ \forall t \in \R^*, \quad
V^*\big(\phi^H_t(x,p)\big) \cap D\phi^H_t(x,p)\cdot V^*(x,p) = \{ 0 \}.$$ 
\subsubsection*{Extremal curves, Tonelli theorem}
We fix a Tonelli Lagrangian $L$.
Let us recall some classical results  (a good reference is \cite{Fathi}). \\
If $\gamma:[a,b]\rightarrow \T^n$ is an absolutely continuous curve, its Lagrangian action is defined by:
$$A_L(\gamma)=\int_a^bL\big(\gamma(t), \gamma'(t)\big)dt.$$

A $C^2$ curve $\gamma: [a, b]\rightarrow \T^n$ is an {\em extremal curve} for $L$ if for each $C^2$ variation $\Gamma: [a, b]\times ]-\varepsilon, \varepsilon[\rightarrow \T^n$   of $\gamma$ with $\Gamma(t,s)=\gamma(t)$ in a neighborhood of $(a, 0)$ and $(b,0)$, we have: $\frac{d{}}{ds}A_L(\Gamma_s)_{s=0}=0$.

A curve $\gamma$ is an extremal curve for $L$ if and only if $(\gamma, \gamma')$ is a solution of the Euler-Lagrange equation
$$
-\frac{d}{dt} \Big( D_v L \big(\gamma (t) , \gamma'(t)\big)\Big)+ D_x L \big(\gamma(t), \gamma'(t)\big) =0 \, . 
$$ 
\begin{theo}[Tonelli] Let $L:T\T^n\rightarrow \R$ be a Tonelli Lagrangian. If $C\in\R$, the subset
$$\Sigma_C=\{ \gamma\in C^{ac}([a, b], \T^n); A_L(\gamma)\leqslant C\}$$
is a compact subset of the set $C^{ac}([a, b], \T^n)$ of absolutely continuous curves endowed with the topology of uniform convergence.

If $x, y\in \T^n$, there exists a curve $\gamma: [a, b]\rightarrow \T^n$ that minimizes the Lagrangian action among all the absolutely continuous curves joining $x$ to $y$. Every such curve is then an extremal curve for $L$.

If $x\in \T^n$ and if $h\in\Z^n$, there exist a loop $\gamma:[a, b]\rightarrow \T^n$ that minimizes the Lagrangian action among all the absolutely curves joining $x$ to $x$ the  homotopy class of whose is $h$. Every such curve is then an extremal curve for $L$.

\end{theo}

A curve $\gamma: [a, b]\rightarrow \T^n$ satisfying the conclusion of  the theorem is called a {\em minimizer}.
There is a similar statement (existence of a minimizer with fixed ends) for the lift $\tilde L$.

\subsubsection*{Palais-Smale condition, coercivity}
A good reference for these notions is \cite{Struwe}.\\
We consider a $C^2$ function $E: \Hc\rightarrow \R$ defined on a Hilbert space $\Hc$. \\
The function $E$ is {\em coercive} if $\displaystyle{\lim_{\| u\|\rightarrow +\infty}E(u)=+\infty}$.\\
The function $E$ satisfies the {\em Palais-Smale condition} if  every sequence $(u_m)\in \Hc$ that is   such that $\displaystyle{\lim_{m\rightarrow \infty} \|DE(u_m)\|=0}$ and $\big(E(u_m)\big)$ is bounded has a subsequence that has a limit.\\

 \subsubsection*{Ma\~n\'e, Mather and Fathi theory, Green bundles}
 See the appendix.

\subsection*{Acknowledgements}

The authors are grateful to Albert Fathi for stimulating conversations and ideas.

\section{On $C^0$ integrability}


\setcounter{theorem}{0}

The main theorem of this section is the following:

\begin{theorem}
Let $L$ be a Tonelli Lagrangian on $T\T^n$. Then $L$ 
has no conjugate points if and only if there is a continuous foliation 
of $T^*\T^n$ by Lipschitz, Lagrangian, flow-invariant graphs.
\end{theorem}
Firstly, observe that it is proved in \cite{arnaud4} that if there is a continuous foliation 
of $T^*\T^n$ by Lipschitz, Lagrangian, flow-invariant graphs, then $L$ has no conjugate points. We just have to prove the converse implication. 

The proof  is not a mere rewriting of Heber's arguments. 
It is achieved in two steps, each using arguments of very different nature. 
We first establish (in section \ref{periodic}) that it is possible to cover the torus by 
periodic \ecs the  period and homotopy class of whose may be fixed arbitrarily. 
To do this, we adapt a technique of metric geometry first 
used by Busemann (see \cite{Bus}, section 32) while he was investigating G-spaces without 
conjugated points. The next step (in section \ref{foliation}) is to show that it is possible 
to associate to every $c \in H^1 (\T^n, \R)$ a Lipschitz graph which is 
Lagrangian, flow-invariant, and formed by orbits which minimize the action 
for the Lagrangian $L - \l$, where $\l$ is a closed $1$-form the  cohomology 
class of whose  is $c$. All this part is inspired by the methods developed in 
\cite{arnaud4}, and uses results of weak KAM theory (see \cite{Fathi}). 
To conclude the proof, we use a topological argument to verify that the 
union of the previously constructed graphs is the whole of the cotangent space.

%
%
%
%
%
%
%
%
%
%
%

\subsection{Some properties satisfied by Lagrangians without conjugate points}

\medskip
In this section, we establish some technical results on Lagrangians without conjugate 
points. They will be used repeatedly in the sequel. 

\smallskip
Let $t > 0$ and $x \in \R^n$. Consider the following map, which is reminiscent of the 
exponential map in Riemannian geometry:  
$$ F : v \in T_x\R^n \longmapsto \pi \circ \fit_{t}(x,v) \in \R^n.$$

\begin{proposition}\label{inj}
The map $F$ is injective.
\end{proposition}

Before proving this, let us establish some consequences: 

\begin{corollary}\label{diffeo}
The map $F$ is a $C^{k-1}$ diffeomorphism.
\end{corollary}

\begin{proof}

This application is of class $C^{k-1}$. It is injective by proposition \ref{inj}, and surjective because of the
  Tonelli theorem. So we only need to check that $F$ is a local diffeomorphism. 
Let $i : v \in T_x\R^n \longmapsto (x,v) \in T\R^n$ be the canonical injection and  
$(x',v') = \fit_{t}(x,v)$. The differential of $F$ at $v$ is 
$$ DF(v) : w \in T_x\R^n \longmapsto D\pi(x',v') \circ D\fit_{t} (x,v) \circ Di(v) 
\cdot w \in T_{x'}\R^n.$$
\noindent
Let $w$ belong to the kernel of $DF(v)$. Then $D\fit_{t} (x,v) \circ Di(v) \cdot w \in 
V(x',v')$ and $Di(v) \cdot w \in V(x,v)$, so that $Di(v) \cdot w = 0$ (because the Lagrangian 
has no conjugate points), and hence $w=0$. This proves that $DF(v)$ is an isomorphism for 
every $v \in T_x\R^n$ and therefore $F$ is a local diffeomorphism.

\end{proof}

Given two points $x$ and $y$ in $\R^n$ and a positive real number $t$, let us introduce 
$$\Ac_t(x,y) = \inf \{ A_L(c), c : [a,b] \longrightarrow \R^n, \ {\rm with} \ 
c(a) = x, \ c(b) = y, \ {\rm and} \ b-a = t \},$$
\noindent
where $A_L(c) = \int _a ^b \LL\big(c(t),c'(t)\big) dt$ is the action of the curve 
$c : [a,b] \longrightarrow \R^n$. Note that the inf (taken over the set of absolutely continuous curves) is actually a min 
(because of the Tonelli theorem). Moreover, thanks to the hypotheses made on $L$, this inf is realized by a $C^2$ curve which is a piece of trajectory of the Euler-Lagrange flow $\fit$. Hence $\Ac_t$ is clearly a continuous 
function, with $\Ac_t(x+r,y+r) = \Ac_t(x,y)$ for every $r \in \Z^n$. As 
a consequence of proposition \ref{inj}, one has: 

\begin{corollary}\label{}
 
Every \ec $c : \R \longrightarrow \R^n$ minimizes the action 
between any two of its points: for every $a , b \in \R$ with $a < b$, 
$\Ac_{b-a}\big(c(a),c(b)\big) = A_L\big(c_{\vert [a,b]}\big)$.

\end{corollary}

Let us now come back to the proof of proposition \ref{inj}.
\begin{proof}[Proof of proposition \ref{inj}]
  To this purpose, we reason by contradiction. Assume therefore that 
there exist two vectors $v_1,v_2\in   T_x \R^n$ such that $F(v_1)=F(v_2)=y$. Let 
us moreover denote by $\gamma_i(s)= \pi\circ \fit_s(x,v_i)$, for $i\in\{1,2\}$ 
and $s\in [0,t]$. 

Using standard techniques (see for example \cite[section 5]{abbondandolo}, we modify $\LL$ and consider a new (periodic) Lagrangian $\LLL$ which 
is still Tonelli, which coincides with $\LL$ on $\R^n\times B(0,R)$, for some $R>0$ 
to be determined later and which is quadratic at infinity (this means there exists 
$R'>R>0$ and a smooth $\Z^n$ periodic function $V$ on $\R^n$ such that for all $x\in \R^n$, there are 
a linear form $l_x$ and a positive definite quadratic form $Q_x$, such that 
$\LLL(x,v)= V(x)+l_x(v)+Q_x(v)$ as soon as $\|v\|>R'$). 

This technical requirement of being quadratic at infinity ensures that the action 
functional $A_{\LLL}$ associated to $\LLL$ is $C^2$ when restricted to the Hilbert 
space $H^1=W^{1,2}$ of curves which are absolutely continuous, with $L^2$ derivative (see for instance \cite[proposition 4.1]{abbondandolo}. Moreover, it then verifies the Palais-Smale condition (see \cite[proposition 4.2]{abbondandolo}). However, by 
construction, any extremal curve for $\LLL$ which remains in $\R^n\times B(0,R)$ 
does not have conjugate points.
It follows from Corollary 4.1 in \cite{Cont} that extremal curves without conjugate 
points are strict local minima of the action functional (the Hessian of the action functional is positive definite at such a curve). In particular, there exists an 
$\e>0$ and $\alpha>0$ such that for any non trivial variation 
$\eta : [0,t]\to \R^n$ in $H^1$ verifying $\eta(0)=\eta(t)=0$ and such 
that $\|\eta\|_1=\e$ then 
$$A_{\LLL}(\gamma_i+\eta)> A_{\LLL}(\gamma_i)+\alpha , \quad i\in \{1,2\}.$$

Let $E$ be the Hilbert space of $H^1$ curves $\eta : [0,t]\to \R^n$ verifying 
$\eta(0)=\eta(t)=0$ equipped with the norm $\|\eta\|_E $ induced by the norm 
$\|\cdot\|_1$ on $H^1$, and consider the restriction of $A_{\LLL}$ to the affine 
space $V=\gamma_1+E=\gamma_2+E$. As already mentioned,
 $A_{\LLL{|V}}$ is $C^2$ and  verifies the Palais-Smale condition. It also inherits coercivity from the superlinearity 
of $\LLL$. 

We now apply the Ambrosetti and Rabinowitz mountain pass lemma (\cite[theorem 6.1 p.109]{Struwe}). 
The lemma asserts that the value 
$C=\inf\limits_{\Gamma}\max\limits_{s\in [0,1]} A_{\LLL}\big( \Gamma(s)\big)$, where $\Gamma$ 
ranges in all the homotopies from $\gamma_1$ to $\gamma_2$ in $V$, is a critical 
value of  $A_{\LLL{|V}}$. More precisely there exists  some curve $\gamma \in V$ realizing the $\inf \max$
in the sense that $\gamma$ is an extremal curve of ${A_{\LLL}}_{|V}$, $A_{\LLL} (\gamma)=C$
and for any $\varepsilon >0$, there is a homotopy $H_{\varepsilon} \in \Gamma$
 such that
 $$
\max_{s\in [0,1]} A_{\LLL} (H_{\varepsilon}(s)) \leqslant C + \varepsilon  
$$ 
and $H_{\varepsilon} ([0,1])$ intersects the ball in $V$ of center $\gamma$ and radius $\varepsilon$.
Obviously, $\gamma$ must not be a strict local minimum of $A_{\LLL{|V}}$ which means 
it contains conjugate points for $\LLL$. In order to reach a contradiction, we 
only need to prove that $\gamma$ is supported in $\T^n\times B(0,R)$ which we 
will see is automatic if $R$ is chosen big enough.

Let $\Gamma_0$ be the linear homotopy: $s\mapsto (1-s)\gamma_1+s\gamma_2$. 
Assume that  $R$ is large enough such that $\Gamma_0$ is supported in $\T^n\times B(0,R)$.
Let 
$$C'=\max_{s\in[0,1]}A_{\LLL}\big( \Gamma_0(s)\big)=\max_{s\in[0,1]}A_{L}\big( \Gamma_0(s)\big)\geqslant C.$$
Note that $C'$ depends only on $L$. The contradiction is now a direct consequence 
of the following lemma:

\begin{lemma}
Let $T>0$ and $M>0$. There exists a constant $R>0$ such that any critical curve 
$\delta : [0,T]\to \T^n$ with action $A_L(\delta) <M$ is $R$-Lipschitz. Moreover, 
this holds true for any other Tonelli Lagrangian which coincides with $L$ on 
$\T^n\times B(0,R)$.
\end{lemma}

\begin{proof}
Let $\delta$ be as in the lemma.
By coercivity of $L$, let $r>0$ be such that $L(x,v)>M/T$ as soon as $|v|>r$. Since  
$A_L(\delta) <M$ there exists $s_0\in [0,1]$ such that $|\delta'(s_0)|\leqslant r$. 
In particular, $\delta$ being an extremal curve, $(\delta,\delta')$ is a trajectory 
of the Euler-Lagrange flow of $L$ which yields that 
$$\forall s\in [0,1],\quad \big(\delta(s),\delta'(s)\big) =\varphi_{s-s_0}\big(\delta(s_0),\delta'(s_0)\big)\in \bigcup_{t\in[-1,1]} \varphi_t\big(\T^n\times \overline{B(0,r)}\big) :=K,$$
which is obviously compact. Therefore, it is enough to take $R$ such that $K\subset \T^n\times B(0,R)$.

The second part of the lemma the follows from the fact that the previous argument 
only depends on the restriction of $L$ to  $\T^n\times B(0,R)$.

\end{proof}
\end{proof}
  
\begin{lemma}\label{tri} 
For any positive real numbers $t$ and $t'$, for any points $x$, 
$y$, and $z$ in $\R^n$, the following inequality holds:     
$$ \Ac_{t+t'}(x,z) \ \leqslant \  \Ac_t(x,y) + \Ac_{t'}(y,z),$$
\noindent
It will be referred to  as the triangular inequality in the sequel. In this inequality, 
equality occurs if and only if $y=c(t)$, where $c : \R \longrightarrow \R^n$ denotes 
the extremal curve with $c(0)=x$ and $c(t+t')=z$.
\end{lemma}

\begin{proof}
Let $c_1: \R \longrightarrow \R^n$ be the \ec with $c_1(0) = x$ and $c_1(t) = y$; 
and $c_2: \R \longrightarrow \R^n$ the \ec with $c_2(t) = y$ and $c_2(t+t') = z$. 
If we concatenate $c_{1 \vert]-\infty, t]}$ and $c_{2 \vert [t, + \infty,]}$, we 
get a curve $\g: \R \longrightarrow \R^n$ with $\g(0) = x$ and $\g(t+t') = z$; hence 
$\Ac_{t+t'}(x,z) \leqslant A_L (\g _{\vert [0,t+t']}) = \Ac_t (x,y) + \Ac_{t '} (y,z)$. 
If we have equality, then $\g_{[0,t+t']}$ is an action-minimizing curve and therefore 
an \ec \!\!. According to proposition 1, $\g$ is equal to $c$. In particular $y = \g(t) = c(t)$.
\end{proof}

\begin{lemma}\label{reg} 
Let $T > 0$ and $r \in \Z^n$. Define a vector field $X$ on $\R^n$ as such: 
if $x \in \R^n$, $X(x) = c'(0)$, where $c$ is the \ec with $c(0) = x$ and $c(T) = x+r$. 
Then $X$ is $\Z^n$-invariant, so   it induces a vector field on $\T^n$, also denoted 
by $X$. This vector field is of class $C^{k-1}$.
\end{lemma}

\begin{proof} 
We shall apply 
the implicit function theorem to the function 
$$\Fc : (x,v) \in T\R^n  = \R^n \times \R^n \longmapsto \pi \circ \fit_T(x,v) - (x + r).$$
\noindent
This will ensure that the map sending $x \in \R^n$ to the unique vector $v = X(x)$ for which 
$\Fc(x,v) = 0$ is of class $C^{k-1}$. All we need to do is to check that the differential 
of $\Fc$ with respect to $v$ at a point $(x,v) \in T\R^n$ is invertible. This is done in 
the same way as in the proof of corollary \ref{diffeo}, using the fact that $L$ has no \cp.
\end{proof}

\subsection{Construction of totally periodic Lagrangian tori}\label{periodic}

\bigskip
The goal of this section is to give a proof of the following:  

\begin{proposition}\label{P8} 
Let $T > 0$ and $r \in \Z^n$. There exists a subset $\Gc_{T,r}$ 
of $T\T^n$ such that

\smallskip \noindent
$(1)$ $\Gc_{T,r}^*=\Lc (\Gc_{T,r})$ is a Lagrangian $C^{k-1}$ submanifold of $T^*\T^n$; 

\smallskip \noindent
$(2)$ $\Gc_{T,r}$ is a flow-invariant graph; 

\smallskip \noindent
$(3)$ $\forall (x,v) \in \Gc_{T,r},\quad \varphi_T(x,v)=(x,v)$; 

\smallskip \noindent
$(4)$ For all $(x,v) \in \Gc_{T,r}$, the \ec $c : s \in [0,T] \longmapsto 
\pi \circ \varphi_s(x,v) \in \T^n$ is a smooth loop with homotopy class 
$r$ and its action does not depend on $(x,v)$.
\end{proposition}

\bigskip
Let us fix $T>0$ and $r \in \Z^n$, and call $f$ the function defined by $f(x) = \Ac_T(x,x+r)$, 
where $x \in \R^n$. As $f : \R^n \longrightarrow \R$ is continuous and $\Z^n$-periodic, we have 
two points $a$ and $b$ in $\R^n$ such that $f(a) = \min\limits_{\R^n}f$ and $f(b) = \max\limits_{\R^n}f$. 
We first show that there is an \ec on the torus with the following properties: it is periodic, 
$T$ being a period; it contains the point $\pr(b)$; and its homotopy class is $r$.  

\begin{lemma}\label{per} 
The \ec $c$ with $c(0) = b$ and $c(T) = b + r$ is invariant under the 
translation of vector $r$. More precisely, we have 
$$\forall s \in \R, \quad c(s+T) = c(s) + r.$$
\noindent
Therefore $\pr \circ c$ is an \ec on the torus that is $T$-periodic, goes through $\pr(b)$ and 
the homotopy class of whose  is $r$.
\end{lemma}

\begin{proof}
Let $x$ be any point in $\R^n$. Using the fact that $\LL$ is $\Z^n$-periodic and the triangular 
inequality, we get    
$$ \Ac_{2T}(b,x) = \Ac_{2T}(b + r , x + r) \leqslant \Ac_{T}(b + r , x) + \Ac_{T}(x , x + r) 
= \Ac_{T}(b + r , x) + f(x).$$
\noindent
We now choose $x = c(2T)$. As $c$ is an \ec which contains the points $c(0) = b$, $c(T) = b + r$ 
and $c(2T) = x$, we have $\Ac_{2T}(b,x) = \Ac_{T}(b,b+r) + \Ac_{T}(b+r,x) = 
f(b) +  \Ac_{T}(b+r,x)$, and hence $\Ac_{T}(b+r,x) = \Ac_{2T}(b,x) - f(b)$. The last inequality 
then becomes
$$ \Ac_{2T}(b,x) \leqslant \Ac_{2T}(b,x) - f(b) + f(x) \leqslant \Ac_{2T}(b,x),$$
\noindent
because $f$ attains its maximum at the point $b$. As a consequence, all the above inequalities 
are in fact equalities. In particular, lemma \ref{tri} tells us that $x = \g(T)$, where $\g$ is the \ec 
with $\g(0) = b + r$ and $\g(2T) = x + r$.

So we get two extremal curves, namely $s \longmapsto \g(s)$ and $s \longmapsto c(s) + r$, that both
go through $b + r$ (at $s = 0$) and $x + r$ (at $s = 2T$). Proposition \ref{inj} implies that they 
are equal; hence $c(T) + r = \g(T)$, i.e. $b + 2r = x$. We then apply the same argument to 
$s \longmapsto c(s) + r$ and $s \longmapsto c(s+T)$: these two \ecs coincide at $s=0$ and at 
$s=T$ (because $b + 2r = x$), 
so they are equal. 
\end{proof}

\begin{lemma}\label{}
The function $f$ is constant.
\end{lemma}

\begin{proof} 
We only need to show that $\max\limits_{\R^n} f = f(b) \leqslant f(a) = \min\limits_{\R^n} f$. An immediate consequence 
of lemma \ref{per} is that $c(nT) = b + nr$ for all integer $n$; therefore if $n \geqslant 1$, we have  
$$ \Ac_{nT}(b , b + nr) = n \Ac_T(b , b + r) = n f(b). $$
\noindent
On the other hand, by use of the triangular inequality, we get that for any $n \geqslant 3$
$$ \Ac_{nT}(b , b + nr) \leqslant \Ac_T(b , a + r) 
+ \sum_{i=1}^{n-2} \Ac_{T}\big(a + ir , a + (i+1)r \big) 
+ \Ac_T\big(a + (n-1)r , b + nr\big).$$
\noindent
These two lines and the fact that $\Ac_T$ is $\Z^n$-invariant imply that 
$$ n f(b) \leqslant \Ac_T(b , a + r) + (n-2)f(a) + \Ac_T(a , b + r).$$
\noindent
Dividing this by $n$ and letting $n$ go to infinity yields $f(b) \leqslant f(a)$.
\end{proof}

\medskip
As the function $f$ attains its maximum everywhere, we may apply lemma \ref{per} to every 
point in $\R^n$. Therefore for every $x \in \T^n$ there exists an \ec $c_x$ containing 
$x$ (we may assume $c_x(0) = x$) which is $T$-periodic, the homotopy class of whose is $r$ 
and whose action $A_L(c_{x \vert [0,T]})$ does not depend on $x$. Remark that this 
curve is unique. For if $c$ and $\g$ are two such curves, we can lift them to $\R^n$. 
Hence we obtain two \ecs $C$ and $\Gamma$, chosen in such a way that $C(0) = \Gamma(0)$. 
As $c(T) = \g(T) = x$ and $c$ and $\g$ belong to the same homotopy class, $C(T) = \Gamma(T)$. 
Using proposition \ref{inj} once again, we conclude that $C = \Gamma$, and that $c = \g$.
 
\medskip
We now define $\Gc_{T,r}$ as the set for all vectors tangent to the curves $c_x$. 
More formally,  
$$\Gc_{T,r} \ = \ \big\{ \big(x,c_x'(0)\big),\quad x \in \T^n \big\}.$$
\noindent
Properties of the curves $c_x$  imply that $\Gc_{T,r}$ satisfies $(2)$, $(3)$ and $(4)$ of proposition \ref{P8}. 
Thanks to lemma \ref{reg}, $\Gc_{T,r}$ is of class $C^{k-1}$. 

It remains to check that $\Gc_{T,r}^*$  is Lagrangian. This requires the use of the Green bundles. Recall 
(see \cite{Cont} and \cite{arnaud1} for details) that if $s \in \R \longmapsto (x,p) = 
\phi_s^H(x_0,p_0) \in T^*\T^n$ is an orbit of the Hamiltonian flow that is free of conjugate points, one may define 
a bundle $G_+$ (called the (positive) Green bundle) by 
$$G_+(x,p) = \lim _{t \longrightarrow +\infty} D \phi_t^H \big( \phi_{-t}^H(x,p) \big) \cdot V^*\big( \phi^H_{-t}(x,p) \big).$$
\noindent
Every $G_+(x,p)$ is a Lagrangian subspace of $T_{(x,p)}T^*\T^n$, and this bundle is invariant by the Hamiltonian flow: 
$D \phi_t^H G_+(x,p) = G_+\big( \phi^H_t(x,p) \big)$ for all $t \in \R$. We now establish that at every 
$(x,p) \in \Gc^*_{T,r}$, the tangent space $T_{(x,p)}\Gc^*_{T,r}$ is equal to $G_+(x,p)$, and is 
therefore Lagrangian.  

We will make use of the following criterion (see \cite{Cont} and \cite{arnaud1}): if $w \in T_{(x,p)}(T^*\T^n)$, 
then $$ w \notin G_+(x,p) \ \Longrightarrow \ \lim_{t \rightarrow + \infty} 
\vv D (\pi \circ \phi^H_{-t})(x,p) \cdot w \vv = + \infty,$$
\noindent
where $\vv \cdot \vv$ denotes the Euclidean norm. Assume $w \in T_{(x,p)}\Gc^*_{T,r}$. As we know that  
$\phi^H_{T \vert \Gc^*_{T,r}} = {\rm Id}_{\vert \Gc^*_{T,r}}$, the same equality holds for $\phi^H_{-nT}$,  
for all integer $n$. Passing to the differential, we get $D \phi^H_{-nT}(x,p) \cdot w = w$, hence 
$D (\pi \circ \phi^H_{-nT})(x,p) \cdot w = D \pi (x,p) \cdot w$ has constant norm. The criterion mentioned 
above implies that $w \in G_+(x,p)$. This proves that $T_{(x,p)}\Gc^*_{T,r} \subset G_+(x,p)$; but these two 
spaces have the same dimension, so they coincide.

\begin{remark}\label{0section}\rm
We do not know if there is an easy way to describe these sets $\Gc_{T,r}$, except in the case where 
$r=0$. In fact, $\Gc_{T,0} = \{ (x,0),\ \  x \in \T^n \}$. To prove this, let $x \in \R^n$ and let $c$ 
be the \ec with $c(0) = c(T) = x$, so that $\big(x,c'(0)\big) \in \Gc_{T,0}$. Denote by $\g$ the \ec with 
$\g(0) = x$ and $\g({T \over 2}) = x$. As a consequence of lemma 4, we have $\g(T) = x$ and hence $\g$ 
and $c$ are equal. In particular, $c({T \over 2}) = \g({T \over 2}) = x$. Repeating the same argument, 
we get $c({T \over 2^n}) = x$ for every $n \geqslant 1$;  therefore $\big(x,c'(0)\big)=(x,0)$.
\end{remark}
\begin{remark}\label{rmkinj}\rm
A remarkable consequence of proposition 2 is that if $c : \R \longrightarrow \T^n$ is an extremal 
curve,  then $c$ is either injective or periodic. For if we can find two real numbers $a$ and $b$ with 
$a < b$ and $c(a) = c(b)$, we lift $c$ to $\R^n$ and so we have an \ec $C : \R \longrightarrow 
\R^n$, and $C(b) = C(a) + r$, with $r \in \Z^n$; lemma \ref{per} now tells us that $c = \pr \circ C$ is 
periodic, and any vector tangent to $c$ belongs to $\Gc_{b-a,r}$.  
\end{remark}
 
\subsection{A continuous foliation of $T^*\T^n$}\label{foliation}

\bigskip
In this section, we will construct a continuous foliation of $T^*\T^n$, with the help of the 
sets $\Gc_{T,r}$ introduced above. The method used here is very close to the one introduced 
in \cite{arnaud4}, where a similar result is proven under the assumption that the tiered Man\'e  
set is the whole cotangent space.

For all $T > 0$ and $r \in \Z^n$, recall that $\Gc_{T,r}^*$ is $\Lc(\Gc_{T,r})$. We first show that 
each of these sets is in fact an Aubry set associated to a cohomology class.

\begin{proposition}\label{}

Let $T > 0$ and $r \in \Z^n$. There is a cohomology class 
$c\in H^1(\T^n,\R)$ such that $\Gc_{T,r}^* = \Ac^*_c = \Nc_c^*$.

\end{proposition}

\medskip
\begin{proof}

The set $\Gc_{T,r}^*$ is a $C^1$ Lagrangian graph in $T^*\T^n$, so it is the graph  $\Gc_\o$ of a 
$C^1$ closed 1-form $\o$. Is it shown in \cite{Fathi} that in this case $\Ac_c^* \subset \Gc_\o 
\subset \Nc_c^*$, where $c$ is the cohomology class of $\o$. We shall prove that these three 
sets are equal. Let $(x,p) \in \Gc_\o = \Gc_{T,r}^*$. As 
$\phi_T^H(x,p) = (x,p)$,  $(x,p)$ belongs to its omega limit set; and it is known that the 
omega limit set of every element of $\Nc_c^*$ is in $\Ac_c^*$. So we have $\Ac_c^* = \Gc_\o$. 
Hence $\pi(\Ac_c^*) = \pi(\Gc_\o) = \T^n$, and this implies $\Ac_c^* = \Nc_c^*$ (see the section on weak KAM theory in the appendix).

\end{proof}

We next describe the sets $\Ac_c^*$, where $c$ is any cohomology class. 

\begin{proposition}\label{}

For all $c \in H^1(\T^n, \R)$, $\Ac_c^*$ is a graph above $\T^n$. Moreover, if $c$ 
and $d$ are two distinct elements in $H^1(\T^n, \R)$, then $\Ac_c^* \cap \Ac_d^* = \varnothing$.

\end{proposition}

\medskip
\begin{proof}

This result is akin to the ones contained in propositions 12 and 13 in \cite{arnaud4}, and the 
proof is roughly the same, so we will only give the main lines of the reasoning, and refer 
to \cite{arnaud4} for more details. 

Let $c \in H^1(\T^n, \R)$. We have to show that $\Ac_c$ is a graph above $\T^n$, i.e. that every $y \in \T^n$ 
is in $\pi (\Ac_c)$.  Let $\l$ be a closed 1-form with cohomology class $c$, and $(x,v) \in \Ac_c$. 
There exist a sequence of real numbers $(T_m)$ with $\lim T_m = + \infty$ and a sequence of \ecs 
$\g_m : \R \longrightarrow \T^n$ such that  
$$ \g_m(0) = \g_m(T_m) = x \ \ {\rm and} \ \ \int _0 ^{T_m} \big(L\big(\g_m(t),\g_m'(t)\big) - \l_{\g_m(t)}\big(\g_m'(t)\big) + \a(c)\big) \ dt 
\longrightarrow 0 \, ,$$
see the part about Aubry sets in the appendix. 

\noindent
According to  remark \ref{rmkinj} at the end of the last section, each $\g_m$ has to be periodic, and $T_m$ 
is a period of $\g_m$, hence $\g_m'(0)$ belongs to $\Gc_{T_m,r_m}$ for a certain $r_m \in \Z^n$. 
Denote the vector for which $(y,w_m) \in \Gc_{T_m,r_m}$ by  $w_m \in T_y\T^n$, and the associated \ec by
$\G_m : s \in \R \longmapsto \varphi_s(y,w_m) \in \T^n$. We know that each 
$\G_m$ is $T_m$-periodic, homotopic to $\g_m$, and that $A_L(\G_{n \vert [0,T_m]}) = 
A_L(\g_{n \vert [0,T_m]})$. These properties imply that  
$$ \G_m(0) = \G_m(T_m) = y \ \ {\rm and} \ \ \int _0 ^{T_m} \big(L\big(\G_m(t),\G_m'(t)\big) - \l_{\G_m(t)}\big(\G_m'(t)\big) + \a(c) \big)\ dt 
\longrightarrow 0,$$
\noindent
and therefore $y \in \pi (\Ac_c)$. 

\medskip
Now suppose that $c$ and $d$ are two cohomology classes with $\Ac_c^* \cap \Ac_d^* \neq \varnothing$. 
Then there exists an action minimizing, flow-invariant, probability measure $\mu$ on $T\T^n$, chosen in such a way that  
its dual measure $\mu^*$ has support included in $\Ac_c^* \cap \Ac_d^*$. If we express the fact 
that $\mu$ is minimal for both $L - \l$ and $L - \eta$, where $\l$ (resp. $\eta$) is a closed 1-form 
with cohomology class $c$ (resp. $d$), and use the convexity of the $\a$ function, we get 
$\a\big({c + d \over 2}\big) = {1 \over 2}\big(\a(c) + \a(d)\big)$. 
Next, let $(x,p) \in \Ac_{{c + d \over 2}}^*$, $(T_m)$ a sequence of 
real numbers with $\lim T_m = + \infty$ and $\g_m : \R \longrightarrow \T^n$ a sequence of \ecs 
with $ \g_m(0) = \g_m(T_m) = x$ and  
$$ \int _0 ^{T_m}\Big( L\big(\g_m(t),\g_m'(t)\big) - {1 \over 2}\Big(\l_{\g_m(t)}\big(\g_m'(t)\big) + \eta_{\g_m(t)}\big(\g_m'(t)\big)\Big) + \a\big({c + d \over 2}\big)\Big) \ dt 
\longrightarrow 0.$$  
\noindent
Using the fact that $\a\big({c + d \over 2}\big) = {1 \over 2}\big(\a(c) + \a(d)\big)$ , we get both limits:   
$$\int _0 ^{T_m} \big(L\big({\g_m(t)},\g_m'(t)\big) - \l_{\g_m(t)}\big(\g_m'(t)\big) + \a(c)\big) \ dt \longrightarrow 0,
$$ 
$$
\int _0 ^{T_m} \big(L({\g_m(t)},\g_m'(t)\big) - \eta_{\g_m(t)}\big(\g_m'(t)\big) + \a(d)\big) \ dt \longrightarrow 0,$$ 
\noindent
which imply that the limit of $\g_m'(0)$ is in $\Ac_c \cap \Ac_d$. This shows that the graph 
$\Ac_{{c + d \over 2}}^*$ is a subset of $\Ac_c^* \cap \Ac_d^*$, so that $\Ac_c^*  = \Ac_d^*$ and 
hence $c = d$.
\end{proof}
We come to the conclusion that the Aubry sets $\Ac_c^*$, with $c$ varying in $H^1(\T^n,\R)$, are 
a family of disjoint, flow-invariant, Lipschitz Lagrangian graphs. All we need to show now is that 
every point of the cotangent space belongs to one of these sets. This is an immediate consequence 
of the next result.

\begin{proposition}\label{}

For every $x \in \T^n$, the map  
$$ F_x : c \in H^1(\T^n, \R) \longmapsto \Ac_c^* \cap T_x^*\T^n \in T_x^*\T^n$$
\noindent
is a homeomorphism.

\end{proposition}

\medskip
\begin{proof}
The map $F_x$ is coercive. Indeed, let $K$ be a compact set in $T_x^*\T^n$ and $X = F_x^{-1}(K)$. We claim that $X$ is bounded: for all $c \in X$, 
$\a(c) = H\big(F_x(c)\big) \in H(K)$, hence $\a(X)$ is included in the compact set $H(K)$; as $\a$ 
is superlinear, this implies that $X$ is bounded.

We then establish that the map $F_x$ is  continuous. 
Let $\wtilde c\in H^1(\T^n, \R)$ be a cohomology class, and $c_m\to \wtilde c$. For all $m>0$ let $\lambda_m$ be a  $C^0$ closed $1$-form of class $c_m$  such that $\Ac_{c_m}^*$ is the graph of $\lambda_m$, and $\wtilde \lambda$ a  $C^0$ closed $1$-form of class $\wtilde c$  such that $\Ac_{\wtilde c}^*$ is the graph of $\wtilde \lambda$. We will show that the sequence $(\lambda_m)$ pointwise converges.
It is a general fact that the Ma\~n\' e sets $\Nc^*_{c}$ vary upper semi-continuously with the cohomology class (see \cite[proposition 13]{ArTi}). Let $y\in \T^n$, the sequence $(y,\lambda_{n,y})\in \Nc^*_{c_m}$ is bounded by the previous argument. But any converging subsequence must converge to an element of $\Nc^*_{\wtilde c}$. Since $\Nc^*_{\wtilde c}=\Ac^*_{\wtilde c}$ is a graph over the base, necessarily, the limit is  $(y,\wtilde\lambda_y)$. Therefore, the sequence converges to $(y,\wtilde\lambda_y)$.

 Evaluating at $x$, we have exactly shown that $F_x(c_m) \to F_x(c)$, hence that $F_x$ is continuous.

It is moreover injective  between two vector spaces of the same dimension. The 
invariance of domain (see \cite{Dold}) states that $F_x$ is an open map. As we have seen that $F_x$ is coercive,  
$F_x$ is proper. 
Thus, $F_x\big(H^1(\T^n, \R)\big)$ is  both open and closed, so it has to be equal to $T_x^*\T^n$. 
This proves that $F_x$ is surjective. Since $F_x$ is open, it is also a homeomorphism.  

\end{proof}

Another consequence of this proposition is that the map  
$$\Fc : (x,c) \in \T^n \times H^1(\T^n,\R) \longmapsto F_x(c) \in T^*\T^n$$ 
\noindent 
is continuous, and therefore the Aubry sets are the leaves of a continuous foliation of 
$T^*\T^n$.

\bigskip

 
 \section{Abundance  of KAM tori}

\subsection{Introduction and statements}\label{ss20}
In this section, we still study the dynamics of a  Tonelli Hamiltonian $H$ on $T^* \T^n = \T^n \times (\R^n)^*$ without conjugate points. We will moreover assume that $H$ is $C^\infty$
(see however Remark \ref{remdiff}). 
The associated Hamiltonian vector-field and Hamiltonian flow will still be denoted by $X_H$ and  $(\phi^H_t)_{t\in \R}$.
As it was proved in the previous section, 
for any $T>0$ and $r \in \Z^n $ there is an invariant Lagrangian graph
$$
\Gc_{T,r}^*:={\cal T}_\infty =\big\{ \big(\theta\,  ,\,  I_\infty + Du(\theta)\big) \ ; \ \theta \in \T^n \big\}
$$
such that all the points of $\Gc_{T,r}^*$ are fixed points of $\phi^H_T$, more precisely
$$
\forall x \in \R^n \, , \quad \wtilde{\phi}^H_T \big(x \, ,\,  I_\infty + Du(x)\big)=\big(x+r \, ,\,  I_\infty + Du(x)\big) \, .
$$ 
Here  $u \in  C^\infty(\T^n , \R)$ is identified with a $\Z^n$-periodic map defined on $\R^n$
 and the flow $(\wtilde{\phi}^H_t)$ on $T^*\R^n=\R^n \times (\R^n)^*$ is the lift of the flow $({\phi}^H_t)$. 
\\[2mm]
Our aim is to prove the existence of a rich family of invariant Lagrangian graphs accumulating to
${\cal T}_\infty$, on which the flow is transitive, conjugated to a linear flow of non-resonant
vector.
\\[2mm]
{\bf Definition}. A vector $\ov{\om} \in \R^n$ is said strongly Diophantine if there are real numbers 
$\gamma >0$ and $\tau $ (necessarily $\geqslant n$)
such that
\be \label{DC}
\forall k \in \Z^n \backslash \{0 \} \, , \  \forall l \in \Z \, , \quad 
|k \cdot  \ov{\om} + l| \geqslant \frac{\gamma}{|k|^\tau} \, . 
\ee
Here is the main result of this section.
\begin{theorem} \label{quasi}
Assume $\ov{\om}$ is  strongly Diophantine. There is $m_0 \in \N \backslash \{ 0 \}$ (depending on $\ov{\om}$) such that, for all
$m \geqslant m_0$, there is a $C^\infty$ Lagrangian embedding $i_m : \T^n \to \TT $ such that
\begin{itemize} 
\item[i)]
$ \  \forall \eta \in \T^n \, , \quad
\phi^H_{mT} \big(i_m(\eta)\big)=i_m(\eta + \ov{\om}) \, $. 
\item[ii)]
Writing $i_m(\eta)=\big(\psi_m (\eta), f_m (\eta)\big)$, $\psi_m$ is a $C^\infty$ diffeomorphism of $\T^n$, isotopic to $id_{\T^n}$
 and
$$
{\cal T}_m= i_m (\T^n)= \big\{ \big(\theta \,  ,\,  (f_m \circ \psi_m^{-1})(\theta)\big) \, ; \ \theta \in \T^n \  \big\} \, .
$$
is a Lagrangian graph; the sequence $({\cal T}_m)$ converges to ${\cal T}_\infty$ in $C^\infty$ topology. 
\item[iii)]
The sequence $(\psi_m)$ 
converges in $C^\infty$ topology to a diffeomorphism $\psi_\infty$ of $\T^n$   (independent of $\ov{\om}$),
isotopic to $id_{\T^n}$.
\item[iv)] The tori ${\cal T}_m$ are  flow-invariant. More precisely,
$ \ i_m^* (X_H)=\dps \frac{r}{T} + \frac{\ov{\om}}{mT} \ $, so that
$$
\forall m \geqslant m_0 \, , \ \forall t\in \R \, , \ 
\forall \eta \in \T^n \, , \quad  {\phi}^H_t \big({i}_m (\eta)\big)
={i}_m \Big(\eta + \frac{t}{T} r + \frac{t}{mT}\ov{\om}\Big).
$$
\end{itemize}
\end{theorem}
\noindent
\begin{remark} \rm
a) The torus ${\cal T}_m$ is the Aubry set ${\cal A}^*_{c_m}$, where $c_m$ is the cohomology
class of the closed $1$-form $f_m \circ \psi_m^{-1}$ \\[2mm]
b) As a by-result, the diffeomorphism  $\psi_\infty$  conjugates the  flow on the completely periodic torus ${\cal T}_\infty$
to the periodic linear flow on $\T^n$ of vector $\dps \frac{r}{T}$. 
\end{remark}
The proof of Theorem \ref{quasi} relies on the following two propositions. The first one provides a 
normal form for $\phi^H_T$ in the neighborhood of ${\cal T}_\infty$ and uses the theorem of
Burago and Ivanov in  \cite {BI}. The second one is a simple application of a KAM theorem for exact symplectic maps. 
\begin{remark} \label{remdiff} \rm
Using a KAM theorem in finite differentiability ( \cite{poschel} , \cite{douady2} , \cite{salamon}), we could prove our result for $H$ of class $C^k$, $k$ large enough. 
More precisely a  survey of the proof would show that, $p \geqslant 1$ being given, if $\overline{\omega}$ satisfies
$(\ref{DC})$ and $k> 2\tau + 7 +p$, then the conclusion of Theorem \ref{quasi} holds with $j_m$ of class $C^p$. 
\end{remark}
{\it Notation.} In what follows, any remainder denoted 
by $O(I^s)$ is a $C^\infty$ map $u$ from  $T^* \T^n$ to some finite dimensional vector-space, 
such that, for any $\alpha \in \N^n$, $\beta \in \N^n$ such that $|\beta| <s$, 
$\ |\partial_\theta^\alpha \partial_I^\beta  u|\leqslant C_{\alpha , \beta} |I|^{s-|\beta|}$ on 
some neighborhood $W$ of $0_{\T^n}=\T^n \times \{0\}$.\\
Moreover, if $B$ is an invertible  operator, the inverse of its transposed operator is denoted by $B^{-T}$.
\begin{proposition} \label{NF}
Under the assumptions of Theorem \ref{quasi} , 
there is a $C^\infty$ symplectic diffeomorphism $G$ of $\TT$ of the form
$$
G(\theta , I)= \big(\psi (\theta)\; ,\;  I_\infty + Du\big(\psi(\theta)\big) + D\psi(\theta)^{-T} I\big) \, ,
$$
where $\psi$ is a diffeomorphism of $\T^n$ isotopic to $id_{\T^n}$,
such that 
${\cal T}_\infty = G(\T^n \times \{ 0 \})$, and
\be \label{normal}
G^{-1} \circ \phi^H_T \circ G (\theta , I)= \big(\theta + \ov{A} I + O(I^2)\;  ,\;   I + O(I^3)\big) \, ,
\ee
where $\ov{A} \in L\big((\R^n)^*, \R^n\big)$ is symmetric positive definite. 
\end{proposition}

\begin{proposition} \label{KAM}
Assume that  $H : T^* \T^n \to \R$ is $C^\infty$ and such that, on a neighborhood of $0_{\T^n}$, 
\be \label{p2nf}
 \phi^H_T  (\theta , I)= \big(\theta + \ov{A} I + O(I^2)\; ,\;  I + O(I^3)\big) \, ,
\ee
where $\ov{A} \in L\big((\R^n)^*, \R^n\big)$ is symmetric non-degenerate.
Let $\ov{\om} \in \R^n \backslash \{ 0 \}$ be strongly Diophantine. 
Then there is $m_1 \in \N^*$ such that, for any $m\geqslant m_1$, there
is a $C^\infty$ Lagrangian embedding $j_m : \T^n \to T^* \T^n$ of the form
$$
j_m(\eta)=\Big(\eta + u_m (\eta) \; ,\;  \ov{A}^{-1} \Big(\frac{\ov{\om}}{m}\Big) + v_m(\eta)\Big) \, ,
$$
with $u_m \in C^\infty (\T^n , \R^n)$, $v_m \in C^\infty \big(\T^n , (\R^n)^*\big)$,  for any $k$:
 $||u_m||_{C^k(\T^n)}=o(1)$,
$||v_m||_{C^k(\T^n)}=o(1/m)$ as $m \to \infty$  , such that
$$
\forall \eta \in \T^n \, , \quad      \phi_{mT}^H \big(j_m (\eta)\big) = j_m (\eta + \ov{\om}) \, .
$$
\end{proposition}

\subsection{KAM meets weak KAM: proof of Theorem \ref{quasi}}

In this subsection we prove Theorem \ref{quasi} from Propositions \ref{NF} and \ref{KAM}.
\\[2mm]
Consider $H_1=H\circ G$, where $G$ is the symplectic map of Proposition 
\ref{NF}. We have $\phi_t^{H_1} =G^{-1}
\circ \phi_t^H \circ G$, hence the Hamiltonian $H_1$ satisfies the assumption of 
Proposition \ref{KAM}.  Consider the  sequence $(j_m)_{m \geqslant m_1}$
of Lagrangian embeddings of $\T^n$ into $\TT$ provided by Proposition \ref{KAM}, with
$$
j_m (\eta)= \Big( \eta + u_m (\eta)\; ,\;  \ov{A}^{-1} \Big(\dps \frac{\ov{\om}}{m}\Big) + v_m(\eta)\Big) \, .
$$
 Since
$ j_m(\eta + k \ov{\om})= \phi_{kmT}^{H_1} \big(j_m (\eta)\big)$, we have
\be \label{dens}
\forall \eta \in \T^n \, , \ \forall k \in \Z \, , \quad
H_1\big(j_m(\eta + k \ov{\om})\big)=H_1\big(j_m(\eta)\big) \, .
\ee
Since $\ov{\om}$ is strongly non-resonant, the orbits of the translation  $\tau_{\ov{\om}}$ are dense in
$\T^n$, hence by $(\ref{dens})$,  $H_1 \circ j_m$ is constant. Since $j_m (\T^n)$ 
is Lagrangian, this implies that 
$$\forall t\in \R, \quad   \phi_t^{H_1} \big(j_m (\T^n)\big) = j_m (\T^n),$$
and we can define on $\T^n$ a smooth flow $(\alpha_t^{H_1})=(j_m^{-1} \circ 
\phi_t^{H_1} \circ j_m)$. Now for any $t\in \R$, $\alpha_t^{H_1}$ commutes 
with $\tau_{\ov{\om}}$, hence ($\ov{\om}$ being non-resonant) $\alpha_t^{H_1}$
itself is a translation of vector $\beta_m (t)$, therefore $ j_m^* (X_{H_1}) = \ov{\beta_m}:=\beta'_m(0)$
is constant on $\T^n$ and $\beta_m(t)=t\bar\beta_m$. We are going to prove that 
\be \label{beta}
\ov{\beta}_m=\frac{r}{T} + \frac{\ov{\om}}{Tm},
\ee
for $m$ large enough. 
Since $\alpha_{Tm}^{H_1} = \tau_{\ov{\om}}$, there is $k_m \in \Z^n$ such that 
\be \label{beta1} Tm \ov{\beta}_m = \ov{\om} + k_m \, . \ee 
Writing $j_\infty (\eta)=(\eta ,0)$ and $b_\infty (\eta)= (j_\infty^* X_{H_1}) (\eta)$, 
since  $||j_m -j_\infty||_{C^k(\T^n)}=o(1)$ for any $k\in \N$,  
we have $||b_\infty-\ov{\beta}_m||_{C^k(\T^n)}=o(1)$ for any $k \in \N$. 
Hence $b_\infty$ is a constant map, equal to $\displaystyle{\ov{\beta}_\infty:= \lim_{m \to \infty} \ov{\beta}_m}$.
Moreover, $\wtilde{\phi}^{H_1}_T (x,0)=(x+r,0)$ (we use here the fact that $\psi$ is isotopic to
$id_{\T^n}$), hence $\ov{\beta}_\infty= \dps \frac{r}{T}$. Thus we obtain
\be \label{beta2}
\ov{\beta}_m=\frac{r}{T} + o(1) \, . 
\ee
We have
\be \label{beta3}
j_m (\eta + T \ov{\beta}_m )=\Big(\eta + T \ov{\beta}_m + u_m (\eta + T \ov{\beta}_m) \; ,\;
\ov{A}^{-1} \Big( \frac{\ov{\om}}{m}  \Big) + o(1/m)\Big) .  
\ee
On the other hand, 
\begin{align}  \label{beta4}
j_m (\eta + T \ov{\beta}_m )=&\phi_T^{H_1} \big(j_m(\eta)\big ) \nonumber \\
=& \phi_T^{H_1} \Big(\eta + u_m (\eta)\;  ,\;  
\ov{A}^{-1} \Big(\dps \frac{\ov{\om}}{m}\Big) + v_m(\eta)\Big) \nonumber \\
=& \Big  (\eta + u_m (\eta)+ \dps \frac{\ov{\om}}{m} + o(1/m)\;  ,\;  
\ov{A}^{-1} \Big(\dps \frac{\ov{\om}}{m}\Big) + o(1/m)\Big) .
\end{align}
Comparing $(\ref{beta3})$ and $(\ref{beta4})$, we derive that there is $l_m \in \Z^n$
such that 
\be \label{modif}
\forall \eta \in \T^n \, , \ T \ov{\beta}_m + u_m (\eta + T \ov{\beta}_m)=
u_m (\eta)+ \dps \frac{\ov{\om}}{m} + l_m + o(1/m)
\ee
Taking the mean-value in $(\ref{modif})$, we obtain
\be  \label{beta5}
T \ov{\beta}_m = \frac{\ov{\om}}{m} + l_m + o(1/m) \, .
\ee
By $(\ref{beta5})$ and $(\ref{beta2})$, $l_m= r + o(1)$, which implies that for $m$
large enough, $l_m=r$. Then $(\ref{beta5})$ gives $Tm \ov{\beta}_m =\ov{\om} + mr + o(1)$,
so that in $(\ref{beta1})$, $k_m=mr+ o(1)$. Hence there is $m_0 \geqslant m_1$ such that for $m\geqslant m_0$, $k_m=mr$ and
$(\ref{beta})$ is satisfied. 

There remains to define $i_m : \T^n \to T^* \T^n$ as
$i_m= G \circ j_m =(\psi_m\, ,\,  f_m)$, with $\psi_m (\eta)= \psi \big(\eta + u_m(\eta)\big)$; the sequence $(\psi_m)$ clearly converges in
$C^\infty$-topology to $\psi_\infty:=\psi$; $j_m(\T^n)$ being a Lagrangian manifold, so is
$i_m(\T^n)=G\big(j_m(\T^n)\big)$, and $i_m(\T^n)$ is a graph because $\psi_m$ is a diffeomorphism
of $\T^n$; i) and iv) of Theorem \ref{quasi} are just a consequence of
$\phi_t^H \circ G = G \circ \phi_t^{H_1}$.

\subsection{The normal form: Proof of Proposition \ref{NF}}


We first introduce a lemma that will be used
in the proof of Propositions \ref{NF} and \ref{KAM}. In what follows, 
denoting by $pr$ the canonical projection from $T^* \R^n$ to
$T^* \T^n$,
we shall identify a map $u$ defined in $T^* \T^n$ with the $\Z^n$-periodic map
$u\circ pr$ defined in $T^* \R^n$.

\begin{lemma}  \label{euler}
Let $K$ be a compact subset of $T^* \T^n$, $U$ be an open bounded neighborhood of $K$ in
$T^* \T^n$, and define $\wtilde{K}=pr^{-1} (K)$, $\wtilde{U}=pr^{-1} (U)$.   Let $X$ be a $C^{k+1}$ vector field on $T^* \T^n$,
with $k \geqslant 1$.
We shall identify $X$ with the vector field $X \circ pr$ on $T^* \R^n$. We assume that $\wtilde{K}$ is preserved 
by the flow $(\wtilde{\varphi}_t)$ of $X$. For any $c_0>0$, there exist $\ep_0>0$ and $c_1>0$ such that
for any $\ep \in (0, \ep_0)$ and for any $u \in C^k \big(T^* \T^n , \R^n \times (\R^n)^*\big) $ such
that $|u-\ep X|_{C^k(U)} \leqslant c_0 \ep^2$, for all $m \in \N$ such that $m\ep \leqslant c_0$,
$$
|(id+u)^m - \wtilde{\varphi}_{m\ep}|_{C^k(\wtilde{K})} \leqslant c_1  \ep \, .
$$

\end{lemma}

\begin{proof}
This result is related to the convergence of the Euler method for the differential equations; however it may be convenient
to provide some details.

For a $C^k$ map $w$ from $\wtilde {U}$ to $\R^n \times (\R^n)^*$, we introduce the notation 
$|w|_k=|w|_{C^k(\wtilde{K})}=\sup\limits_{z\in \wtilde{K} , |\alpha| \leqslant k} |D^\alpha w (z)|$.   We shall use the following (rough) estimates,
where $C(s)$ stands for any positive non-decreasing function of the positive variable $s$. \\[1mm]
i) For $\rho \in C^{k} \big(T^* \T^n, \R^n \times (\R^n)^*\big)$ and $a \in C^k(\wtilde{U}, T^* \R^n)$
with $a(\wtilde{K}) \subset \wtilde{U}$,
$$
|\rho \circ a |_k \leqslant C ( |a|_k) |\rho |_{C^k(\wtilde{U})}  \, . 
$$
ii) For $w \in C^{k+1} \big(T^* \T^n , \R^n \times (\R^n)^*\big)$  and  $a,b \in C^k( \wtilde{U}, 
T^* \R^n)$ 
such that 
for all $ x \in \wtilde{K}$, $[a(x),b(x)] \subset \wtilde{U}$,  
$$
| w \circ a - w \circ b|_k \leqslant C( |a|_k + |b|_k) |a-b|_k |w|_{C^{k+1}(\wtilde{U})}  
$$
iii) For $\ep \geqslant 0$ small enough,
$$
|\wtilde{\varphi}_\ep - id |_{C^{k+1} (\wtilde{U})} \leqslant C \ep    \quad , \quad
|\wtilde{\varphi}_\ep - id - \ep X|_{C^k (\wtilde{U})} \leqslant C \ep^2 \, .  
$$
Let us fix $\delta>0$
such that $\delta < {\rm dist} (\wtilde{K}, \partial \wtilde{U})$.
Under the assumptions of Lemma \ref{euler}, let us introduce, for $m \in \N$, 
$\ \alpha_m=|(id+u)^m -\wtilde{\varphi}_{m \ep}|_k \, $. As long as $\alpha_m \leqslant \delta$ and $m\ep \leqslant c_0$, we have
\begin{align}  \label{estalpha1} 
\alpha_{m+1}&= |(id+u)^{m+1} -\wtilde{\varphi}_{(m+1) \ep}|_k  \nonumber \\
&\leqslant  |(id +u -\wtilde{\varphi}_\ep) \circ (id+u)^m|_k +
|(id+u)^m-\wtilde{\varphi}_{m \ep}|_k  \nonumber \\
&\qquad \qquad \qquad \qquad \qquad\qquad \quad  +  |(\wtilde{\varphi}_\ep-id)\circ (id+u)^m   - (\wtilde{\varphi}_\ep-id)\circ
\wtilde{\varphi}_{m\ep} |_k   
\end{align} 
Since $|(id+u)^m -\wtilde{\varphi}_{m \ep}|_0  \leqslant \alpha_m \leqslant \delta$ and $\wtilde{\varphi}_{m\ep} (\wtilde{K}) 
\subset \wtilde{K}$, there holds $$\ \forall x \in \wtilde{K} \, , \ [\wtilde{\varphi}_{m\ep}(x) , (id +u)^m (x)] \subset \wtilde{U}.$$
Hence, by i) and iii), 
\begin{align} \label{est1}
|(id +u -\wtilde{\varphi}_\ep) \circ (id+u)^m|_k & \leqslant C(|(id +u)^m|_k) | id +u -\wtilde{\varphi}_\ep |_{C^k (\wtilde{U})}
\nonumber \\
& \leqslant C(\alpha_m + | \wtilde{\varphi}_{m \ep}|_k) (C \ep^2 + |u-\ep X|_\cku) \nonumber \\
&\leqslant C(\delta , c_0) \ep^2  \,  , 
\end{align}
and by ii),
\begin{align} \label{est3}
|(\wtilde{\varphi}_\ep-id)\circ (id+u)^m   - (\wtilde{\varphi}_\ep-id)\circ
\wtilde{\varphi}_{m\ep} |_k   &\leqslant C( |(id +u)^m|_k + |\wtilde{\varphi}_{m\ep}|_k) \, 
| (id+u)^m- \wtilde{\varphi}_{m\ep}|_k \nonumber \\
& \quad \quad  \quad \quad   \quad \quad \quad \quad   \quad \quad  \times |\wtilde{\varphi}_{\ep} -id|_\ckpu  \nonumber \\
&\leqslant C(\delta , c_0) \ep \alpha_m \, . 
\end{align}
From (\ref{estalpha1}), (\ref{est1}) and (\ref{est3}), we obtain
\be \label{estalpha}
\alpha_{m+1} \leqslant C_1(\delta , c_0) \ep^2 + (1+ C_2(\delta , c_0)\ep ) \alpha_m
\ee
We can assume $C_2(\delta,c_0) \geqslant 1$ without loss of generality. 
Setting $\Gamma= \dps \frac{C_1(\delta,c_0)}{C_2(\delta,c_0)}$, using $\alpha_0=0$, we derive from 
$(\ref{estalpha})$ 
$$
\alpha_m + \Gamma \ep \leqslant \big(1+ C_2(\delta,c_0) \ep\big)^m \Gamma \ep 
$$
as long as $\big(1+ C_2(\delta,c_0) \ep\big)^m \Gamma \ep  \leqslant \delta$ and $m\ep \leqslant c_0$. Hence 
let us define $c_1=\sup\limits_{\ep \in (0,1]} (1+ C_2(\delta , c_0) \ep)^{c_0/\ep} \Gamma$ and choose
$\ep_0 >0$ such that iii) holds for $\ep \in [0, \ep_0]$  and $\ep_0 \leqslant \min (1, \delta / c_1)$. 
Note that $c_1$ and $\ep_0$ depend only on $\delta , c_0$ and the vectorfield $X$. 
For $m\ep \leqslant c_0$, and $\ep \leqslant 
\ep_0 $ then  $\alpha_m \leqslant c_1 \ep$. 
\end{proof}
\begin{proof}[Proof of  Proposition \ref{NF}] We introduce a first symplectic change of variable, with
$$
G_0(\theta , I)= \big(\theta \;  ,\;  I + I_\infty + Du (\theta)\big) \quad ,  \quad H_0=H \circ G_0
$$
We have  $\phi^H_t \circ G_0 = G_0 \circ \phi^{H_0}_t$. Note that, since $G_0$ preserves the fibers 
and its restriction to each fiber is a translation,  
 $H_0$  is a Tonelli Hamiltonian without conjugate points, as $H$.  The Hamiltonian flow
associated to $H_0$ preserves 
$0_{\T^n} =\T^n \times \{ 0\}$ and satisfies 
$$
\forall \theta \in \T^n \, , \quad  \phi^{H_0}_T (\theta , 0)=(\theta ,0) \, .
$$
This implies that the differential of $\phi^{H_0}_T$ at $(\theta, 0)$ takes the form
$$
{D\phi^{H_0}_T}{(\theta,0)} [\Delta \theta , \Delta I]= \big( \Delta \theta + A(\theta) \Delta I \; , \;  Q(\theta) \Delta I \big)
$$
with $A(\theta)=\partial_I \phi^{H_0}_{T,1}(\theta ,0) \in L\big((\R^n)^*, \R^n\big)$, and 
$ Q(\theta)=\partial_I \phi^{H_0}_{T,2} (\theta ,0) \in  L\big((\R^n)^*, (\R^n)^*\big)$.   
The linear map ${D\phi^{H_0}_T}{(\theta,0)}$ being symplectic, $Q(\theta)=id_{(\R^n)^*}$ and $A(\theta)$ is symmetric. 

Let us justify that $A(\theta)$ is positive definite. Given $\theta \in \T^n$, since 
$\phi^{H_0}_t (0_{\T^n}) \subset 0_{\T^n}$, we can write for $t \in [0,T]$, 
$$
D\phi^{H_0}_t (\theta , 0) (\Delta \theta , \Delta I)=(P_t \Delta \theta + A_t \Delta I \; ,\; 
Q_t \Delta I) \, ,
$$
with $P_0=id_{\R^n}$, $Q_0=id_{(\R^n)^*}$, $A_0=0$. Moreover, $D\phi^{H_0}_t (\theta ,0)$ being symplectic,
$P_t^T Q_t = id_{(\R^n)^*}$ and $S_t=Q_t^T A_t$ is in the space $ L_s\big((\R^n)^* , \R^n\big)$ of symmetric
linear maps from $(\R^n)^*$ to $\R^n$. We have
$$
\frac{d S_t}{dt}_{|t=0}=  \frac{d A_t}{dt}_{|t=0} =\frac{d }{dt} \partial_I \phi_{t,1}^{H_0}(\theta , 0)_{|t=0}=
\partial_I X_{H_0,1} (\theta , 0) = \partial^2_I H_0 (\theta ,0) \, .
$$
Thus $\dps \frac{d S_t}{dt}_{|t=0}$ is symmetric positive definite. Moreover, since 
$H_0$ is without conjugate points, $A_t$ is invertible for all $t\in (0,T]$, and
so is $S_t=Q_t^T A_t$. Since $S_t$ is positive definite for small $t>0$, we can conclude
that $S_t$ remains symmetric positive definite for all $t\in (0,T]$. In particular, 
$A(\theta)=S_T$ is symmetric positive definite. 
 
Let us consider the Taylor expansion of $\phi^{H_0}_T$ with respect to $I$, in the neighborhood of 
$0_{\T^n}=\T^n \times \{ 0\}$,
\be \label{taylor}
\phi^{H_0}_T (\theta , I)= \big(\theta  +  A(\theta) I + O(I^2) \; ,\;   I + \frac{1}{2} \la 
B(\theta) I, I \ra + O(I^3)\big) \, ,  
\ee
where $B(\theta )= \partial^2_I \phi^{H_0}_{T,2} (\theta, 0) \in (\R^n)^* \otimes L_s \big((\R^n)^* , \R^n\big)$. 

For $\ep >0$, define ${\cal R}_\ep : T^* \T^n \to T^* \T^n$ by 
$$
{\cal R}_\ep (\theta , I)= (\theta , \ep I) \, ,
$$
and let $\Phi : (0, 1] \times T^* \T^n \to T^* \T^n $ be defined by $\ \Phi(\ep, .)= 
{\cal R}_\ep^{-1} \circ  \phi^{H_0}_T \circ {\cal R}_\ep  $. Since ${\cal R}_\ep^* (I .d\theta)= \ep I. d\theta$, 
the map $\Phi(\ep, .)$ is exact symplectic as $\phi^{H_0}_T$ is. 
By $(\ref{taylor})$, $\Phi$ can be
smoothly extended to a map on $[0, 1] \times T^* \T^n$ with $\Phi(0, .)=id_{T^* \T^n}$ and, as $\ep$ tends to $0$,
\be \label{taylorep}
\Phi (\ep , \theta , I)=\big(\theta + \ep A(\theta) I + O(\ep^2)\;  ,\;  I + \frac{\ep}{2}
\la B(\theta) I , I \ra + O(\ep^2)\big) \, . 
\ee
We have
$$
Y(\theta , I) := \frac{d}{d\ep} \Phi(\ep, \theta , I)_{| \ep =0} = \big(A(\theta) I\, ,\,  \frac{1}{2} \la B(\theta) I \, ,\,  I \ra\big) \, .
$$
Now, since $\Phi(\ep , .)$ is exact symplectic for all $\ep \geqslant 0$, $Y$ is a Hamiltonian
vector-field: there is $F \in C^\infty ( T^* \T^n , \R)$ such that $Y=X_F$. We have then 
$\partial_\theta F(\theta , 0) =0$, hence $F(\theta , 0)$ does not depend on $\theta$ and we may impose
$F(\theta, 0)=0$. Since $\partial_I F(\theta, I)=A(\theta) I$, we have $F(\theta , I)=
\dps \frac{1}{2} \la A(\theta ) I , I \ra$, hence $ B(\theta )  = - (\partial_\theta A) (\theta)$. 
As a result,
\be \label{taylorplus}
\phi^{H_0}_T (\theta , I)= \big(\theta  +  A(\theta) I + O(I^2) \; ,\;  I - \frac{1}{2} \partial_\theta \la 
A(\theta) I, I \ra + O(I^3)\big) \, . 
\ee
Note that the flow $(\phi_t^F)$ spanned by $X_F$ is the geodesic Hamiltonian  flow 
associated to the Riemannian metric on $\T^n$ 
$$
g (\theta)[U,V]=\la A(\theta)^{-1} U ,V \ra
$$
We are going to prove that the geodesic flow of $g$ has no conjugate points. 
Arguing by contradiction, we assume the contrary. Then there are two points
$x$ and $y$ of $\R^n$ that are connected in time $S$ by a non-minimizing
geodesic path $\gamma_1$. We may assume that $\gamma_1$ is non-degenerate, using
the fact that along a geodesic path, the points that are conjugate to the
starting point are isolated.  Now there is also a minimizing geodesic path
$\gamma_2 : [0,S] \to \T^n$ connecting $x$ and $y$, and we may assume 
that $\gamma_2$ is non-degenerate: if not, we just substitute 
$S-\delta$ to $S$ and $\gamma_2 (S-\delta)$ to $y$. For $\delta \in (0,S)$,
${\gamma_2}_{|[0,S-\delta]}$ is a non-degenerate minimizing geodesic path, 
and if $\delta$ is small enough, by the implicit function theorem there is also
a non-degenerate geodesic path $\ov{\gamma_1}$ (close to ${\gamma_1}_{|[0, S-\delta]}$)
connecting $x$ and $\gamma_2 (S-\delta)$ in time $S-\delta$. 

So we have two points $x$ and $y$ of $\R^n$ connected in some time $S>0$
by two distinct non-degenerate geodesic paths.  In other words, there are
$I_1 \neq I_2$ such that 
$$
 \wtilde{\phi}^F_{S,1} (x,I_1) =  \wtilde{\phi}^F_{S,1} (x,I_2) = y \, ,
$$   
with $\partial_I \wtilde{\phi}^F_{S,1} (x,I_1)$ and $\partial_I \wtilde{\phi}^F_{S,1} (x,I_2)$ invertible.

Let us consider the compact subset of $T^*\T^n$ 
$$
K= \{ (\theta , I) \in T^* \T^n  \ : \  \la A(\theta) I , I \ra \leqslant R \, \}
$$
where $R$ is such that $(x, I_i) \in {\rm int} (\wtilde{K})$. Note that 
$\wtilde{K}$ is invariant by $(\wtilde{\phi}^F_t)$. Let us denote by
$\ov{\Phi}_\ep : T^* \R^n \to T^* \R^n$ the lift of $\Phi(\ep, .)$ such that
$\ov{\Phi}_\ep -id_{T^* \R^n}=O(\ep)$. Thus we have
$$
\ov{\Phi}_\ep (x, I)= \big(x+ \ep A(x) I + O(\ep^2)\, , \,  I - \frac{\ep}{2} \partial_x \la A(x) I , I \ra + O(\ep^2)\big) = (x,I) + \ep X_F(x,I) + O(\ep^2) \, . 
$$
By Lemma \ref{euler}, there are $\ep_0 >0$ and $c_1 \geqslant 0$ such that 
$$
\forall \ep \in (0, \ep_0] \, , \ \forall m \in \N \cap [1, 2S/\ep ] \, , \quad
| \ov{\Phi}_\ep^m - \wtilde{\phi}^F_{m\ep } |_{C^1 (\wtilde{K})} \leqslant c_1 \ep \, . 
$$
Choosing $\ep_m =S/m$, for $m$  large enough, 
$\ | \ov{\Phi}_{\ep_m}^m - \wtilde{\phi}^F_{S } |_{C^1 (\wtilde{K})} \leqslant c'/m $.
Hence, by the implicit function theorem, there are $p$ and 
$\ov{I_1}, \ov{I_2} \in (\R^n)^*$ with $\ov{I_1} \neq \ov{I_2}$, $\ov{I_i}$ close to $I_i$,
such that
\be  \label{noninj}
\pi \circ \ov{\Phi}_{\ep_p}^p (x, \ov{I_1}) = \pi \circ \ov{\Phi}_{\ep_p}^p (x, \ov{I_2})=y \, .
\ee
Define $\tau_r : T^* \R^n \to T^* \R^n$ by $\tau_r (x,I)=(x+r,I)$; $\tau_r$ commutes with
$\ov{\Phi}_{\ep}$, and 
$${\cal R}_\ep^{-1} \circ \wtilde{\phi}^{H_0}_T \circ {\cal R}_\ep= \tau_r \circ 
\ov{\Phi}_{\ep}.$$
 Hence $ {\cal R}_\ep^{-1} \circ \wtilde{\phi}^{H_0}_{p T} \circ 
{\cal R}_\ep =\tau_{p r} \circ \ov{\Phi}_{\ep}^p$.
By $(\ref{noninj})$, we obtain
$$
\pi \circ \wtilde{\phi}^{H_0}_{p T} (x, \ep_p \ov{I_1})=\pi \circ \wtilde{\phi}^{H_0}_{p T} (x, \ep_p \ov{I_2}) \, ,
$$
which contradicts the fact that $(\pi \circ \wtilde{\phi}^{H_0}_{p T} )_{|\pi^{-1} (x)}$ is injective.
As a conclusion, the geodesic flow of $g$ has no conjugate points. 

By the
theorem of Burago and Ivanov, this implies that $g$ is flat: there
exists a $C^\infty$ diffeomorphism $\psi$  of $\T^n$  isotopic to the identity  and a positive definite
$\ov{B} \in L_s\big(\R^n , (\R^n)^*\big)$  such that 
$$
\forall \theta \in \T^n \, , \ \forall (U,V) \in \R^n \times \R^n \, , \quad
\big\la A\big(\psi(\theta) \big)^{-1} D\psi  (\theta) \cdot U\;  ,\;  D\psi  (\theta) \cdot V \big\ra
=\la \ov{B} U , V\ra \, . 
$$
Hence we have
\be \label{AA}
\forall \theta \in \T^n \, , \quad  D\psi  (\theta)^{-1} A \big(\psi(\theta)\big) D\psi  (\theta)^{-T}= \ov{A}:= \ov{B}^{-1} \, .
\ee
Let us consider the symplectic diffeomorphism of $\TT$
\be \label{g1}
G_1(\theta , I)= \big(\psi(\theta)\; ,\;  D\psi (\theta)^{-T}  I\big)
\ee
and the Hamiltonian $H_1=H_0 \circ G_1$. Then $H_1=H \circ G$, with
$$
G(\theta , I)= G_0 \circ G_1 (\theta , I)=\big(\psi(\theta)\; ,\;  I_\infty + Du \big(\psi(\theta)\big) + 
D\psi (\theta)^{-T}  I \big)
$$
and
$$
G^{-1} \circ \phi_T^H \circ G= G_1^{-1} \circ \phi_T^{H_0} \circ G_1 = \phi_T^{H_1} \, .
$$
We have $\phi_T^{H_1} (\theta , 0)= (\theta , 0)$. By $(\ref{g1})$ and an easy computation,
$$
\partial_I \phi_{T,1}^{H_1} (\theta , 0)=\partial_I (G_1^{-1} \circ \phi_T^{H_0} \circ G_1)_1 (\theta , 0)= 
D\psi  (\theta)^{-1} A \big(\psi(\theta)\big) D\psi  (\theta)^{-T}= \ov{A}
$$
With the same arguments that we had used to obtain $(\ref{taylorplus})$,
we can conclude that 
$$
\phi_T^{H_1} (\theta, I)=\big(\theta + \ov{A} I + O(I^2)\;  ,\;  I + O(I^3)\big) \, . 
$$
The proof of Proposition \ref{NF} is complete. 
\end{proof}

\subsection{Proof of Proposition \ref{KAM}}

As in the proof of Proposition \ref{NF}, we introduce 
$\ \Phi_\ep =  {\cal R}_\ep^{-1} \circ \phi_T^H \circ {\cal R}_\ep$, where ${\cal R}_\ep (\theta , I)=(\theta , \ep I)$. 
Then $\Phi_\ep$ is an exact symplectic diffeomorphism
of $\TT$ and by $(\ref{p2nf} )$, one has
$$
\Phi_\ep (\theta, I)= \big(\theta + \ep \ov{A} I + O(\ep^2) \; ,\;  I + O(\ep^2)\big) \, . 
$$
By Lemma \ref{euler}, the exact symplectic diffeomorphism ${\cal U}_m = \Phi_{1/m}^m$ satisfies
$$
{\cal U}_m (\theta , I)=\big(\theta + \ov{A} I + O(1/m)\; ,\;  I+O(1/m)\big) \, .
$$
Here $O(1/m)$ means that for any $k\in \N$,  and any compact neighborhood $K$ of $0_{\T^n}$,
the $C^k$ norm of the remainder restricted to $K$ is less or equal to $C_{k,K}/m$. 

The matrix $\ov{A}$ being non-degenerate, by a  theorem essentially due to J. Moser 
(see for instance Theorem 1.2.3. in \cite{bost} , or \cite{zehnder}), for any strongly Diophantine $\ov{\om}$, for $m$ large enough
(depending on  $\ov{\om}$) there is a Lagrangian
embedding  $\rho_m : \T^n  \to \TT$ with $\rho_m(\eta)=\big(\eta + o(1)\;  ,\;   \ov{A}^{-1} \ov{\om} + o(1)\big)$ such that
$$
\forall \eta \in \T^n \, , \quad  {\cal U}_m \big(\rho_m (\eta)\big)=\rho_m (\eta + \ov{\om}) \, .
$$
Here $u_m=o(\alpha_m)$ means that for any $k \geqslant 0$, $\alpha_m^{-1} ||u_m||_{C^k (\T^n)}$ tends to $0$
as $m \to \infty$.
Now let $j_m(\eta)=  {\cal R}_{1/m} \big(\rho_m (\eta)\big)$.  Then 
$$
\phi^H_{mT} \big(j_m(\eta)\big)=j_m (\eta + \ov{\om}) \, . 
$$
Moreover, we have
$$
j_m (\eta ) = \Big(\eta + o(1)\;  ,\;  \frac{\ov{A}^{-1} \ov{\om}}{m} + o(1/m)\Big) \, .
$$
This completes the proof of Proposition \ref{KAM} .

\subsection{Abundance of strictly ergodic invariant tori}
Following \cite{FatHe}, we will say that a set $K\subset T\T^n$ that is invariant by a Tonelli flow $(\phi_t^H)$ is {\em strictly ergodic} if:
\begin{enumerate}
\item[--] the restricted flow $(\phi^H_{t|K})$ has a unique invariant Borel probability measure; this measure is denoted by $\mu$;
\item[--] the support of $\mu$ is $K$.
\end{enumerate}
If $K$ is strictly ergodic, then the restricted flow $(\phi^H_{t|K})$ is minimal (i.e. has no nontrivial invariant closed subset).

A.~Fathi and M.~Herman proved in \cite{FatHe} that if $(X, d)$ is a compact metric space, then the set of its strictly ergodic homeomorphisms is a $G_\delta$-subset of the set of its homeomorphisms endowed with its usual $C^0$ topology. 

If $\Tc$ is one of the tori given by theorem \ref{quasi}, then the homeomorphism $(\phi^H_{1|\Tc})$ is strictly ergodic. As the set of the tori given by theorem \ref{quasi} is dense in the set of the tori given by theorem \ref{th1}, we deduce:

\begin{theorem}
Let $(\phi^H_t)$ be a $C^\infty$ Tonelli flow of $T^*\T^n$ with no conjugate points  and let $\mathfrak F$ be the continuous foliation in invariant Lagrangian tori that is given by theorem \ref{th1}. Then there is a dense $G_\delta$ subset $\Gc$ of $\mathfrak F$ such that, for every $\Tc \in\Gc$, then $\phi^H_{1|\Tc}$ is strictly ergodic.
\end{theorem}


\section{Zero entropy}
 


   Since the arguments of this section are more general, $M$ will denote a closed, smooth manifold. Obviously, the case of the torus $\T^n$ is included in this setting. Recall that in this setting, the notion of $C^0$ integrability persists. A Hamiltonian $H$ on $T^*M$ is said to be $C^0$ integrable if $T^*M$ is partitioned by Lipschitz, invariant Lagrangian  submanifold which are Hamiltonianly isotopic to the $0$-section. As proved in \cite{ArBi,Ber}, in the case of Tonelli Hamiltonians, those submanifolds must be graphs.
   
          \begin{theorem}\label{T1}
   Let $H:T^*M\rightarrow \R$ a $C^3$ Tonelli Hamiltonian that is $C^0$ integrable. Then for every invariant Borel probability measure, the Lyapunov exponents are zero.
    \end{theorem}
Because of Ruelle's inequality (see \cite{ruelle1}), this implies:
\begin{corollary}
The metric entropy of every Borel probability measure that is invariant by a $C^0$ integrable Tonelli Hamiltonian is zero.
\end{corollary}
   Because of the variational principle (see \cite{katok} p.181 for example), this implies:
   \begin{corollary}
   The topological entropy of the Hamiltonian flow of every $C^0$ integrable Tonelli Hamiltonian is zero.
   \end{corollary} 
   \begin{remark}\label{previous}\rm   Even if the invariant foliation is not $C^1$, if all the invariant $C^0$ Lagrangian graphs are assumed to be everywhere differentiable, the result is straightforward: let us assume that some ergodic Borel probability measure has at least one non-zero Lyapunov exponent. Because the dynamics is symplectic, then the number $d\geqslant 1$ of  positive Lyapunov exponents is equal to the number of negative Lyapunov exponents.
    Then the support of this ergodic measure   is contained in some invariant differentiable $C^1$ Lagrangian graph $\Gamma$. Let $\zeta \in T^*M$ be a regular point for $\mu$. There exist two $d$-dimensional embedded open disks $D^u$ and $D^s$ that contain $\zeta $,  the first one in the unstable set of $\zeta $ and the second one in its stable set (see theorem 6.1. in \cite{ruelle2}). Moreover, $T_\zeta D^u$ and $T_\zeta D^s$ are transverse and such that $T_\zeta D^u\oplus T_\zeta D^s$ is a symplectic subspace of $T_\zeta (T^*M)$.\\
Then $D^u\cup D^s\subset \Gamma$. Indeed, let us consider $\zeta '\in D^u\cup D^s$ and let us assume that $\zeta '\notin\Gamma$. Then there exists another invariant Lagrangian graph $\Gamma_1$ such that $\zeta '\in\Gamma_1$ and $\Gamma\cap\Gamma_1=\varnothing$. If for example $\zeta '\in D^s$, we obtain:
\begin{itemize}
\item $\displaystyle{ \lim_{t\rightarrow +\infty} d\big(\phi_t^H(\zeta ), \phi^H_t(\zeta ')\big)=0}$;
\item$\forall t\in\R ,\quad     \phi_t^H(\zeta )\in\Gamma, \   \phi_t^H(\zeta ')\in\Gamma_1$.
\end{itemize}
Because $\Gamma$ and $\Gamma_1$ are compact, this is impossible. \\
We deduce that $D^u\cup D^s\subset \Gamma$ and then $T_\zeta D^u\oplus T_\zeta  D^s\subset T_\zeta \Gamma$. But $T_\zeta D^u\oplus T_\zeta  D^s$ is a $2d$-dimensional symplectic subspace and $T_\zeta \Gamma$ is Lagrangian, hence this cannot happen.
\end{remark}

When a $C^0$ Lagrangian  invariant graph  is  not assumed to be differentiable, the symplectic product of two of its tangent vectors (in a generalized sense) may be non-zero: consider what happens at the hyperbolic critical point of the Hamiltonian $H: \T\times \R\rightarrow \R$ defined by $H(\theta, r)=\frac{1}{2}r^2+\cos(2\pi\theta)$: one separatrix is a Lipschitz Lagrangian invariant graph such that the tangent cone at the critical point contains the stable and unstable tangent lines, and the symplectic product of one stable vector with an unstable one is non-zero.

\begin{remark}\rm To prove Theorem \ref{T1},    proposition \ref{P1} and subsection \ref{ss21} are useless. But we prefer  to give a particular proof in the cases of non-uniform hyperbolicity and of atomic measures because the arguments are simpler in these cases, even if they  do  not work for the general case.
\end{remark}
\subsection{Proof of theorem \ref{T1} in the case of atomic measures}
We assume that $H$ is a $C^0$ integrable Tonelli Hamiltonian of $T^*M$.  Let us prove:
\begin{proposition}\label{P1}
Let $H:T^*M\rightarrow \R$ be a $C^k$, Tonelli Hamiltonian that is $C^0$ integrable (with $k\geqslant 2$). Then the set of all critical points of $H$ is a $C^{k-1}$ Lagrangian graph.
\end{proposition}
\begin{remark}\rm
Proposition \ref{P1} is very similar to what was noted in remark \ref{0section}. However, for completeness, we will recall the main arguments in this setting.
\end{remark}

\begin{corollary}\label{C3}
The Lyapunov exponents of every invariant measure supported by a critical point are zero. 
\end{corollary} 
 \begin{proof}[Proof of proposition \ref{P1} and corollary \ref{C3}] As $H$ is $C^2$, convex and superlinear in the fiber direction, for every $x\in M$ there exists a unique $p\in T_x^*M$ such that $(x, p)$ is a critical point of $H(x, .)$, i.e. such that $\partial_p H(x, p)=0$ (we write this equation in charts). Because the Hessian $\partial ^2_{p}H$ is non-degenerate, we can use the implicit function theorem to deduce that $\Sigma=\{ (x, p), \ \  \partial_p H(x, p)=0\}$ is the graph of a $C^{k-1}$-function.
 
 Let us now prove that all the points of $\Sigma$ are critical points of $H$. If not, let $\zeta \in\Sigma$ such that $X_H(\zeta )\not=0$. Then $X_H(\zeta )=\big(0, -\partial_x H(\zeta )\big)$ is a non-vanishing vertical vector. But in the section 3.5. of \cite{arnaud1}, it is proved that $X_H(\zeta )$ is contained in the two Green bundles at $\zeta $, that are transverse to the vertical. Hence $X_H(\zeta )=0$ and $\zeta $ is a critical point of $H$. 
 
 Let us notice that $(\phi^H_{t|\Sigma})$ is just the  identity map. Hence if $v\in T_\zeta \Sigma$ is any vector tangent to $\Sigma$, its orbit $(D\phi^H_t.v)$ is constant and then by the dynamical criterion  given in 3.5 of \cite{arnaud1}, is in the two Green bundles $G_-(\zeta )$ and $G_+(\zeta )$. As $T_\zeta \Sigma$, $G_-(\zeta )$ and $G_+(\zeta )$ have the same dimension, we deduce that $T_\zeta \Sigma=G_-(\zeta )=G_+(\zeta )$ is Lagrangian and then $\Sigma$ is Lagrangian. 
 
 In a symplectic  linear chart where we choose the first coordinates in $T_\zeta \Sigma$, the matrix of $D\phi^H_t(\zeta )$ is a symplectic matrix $A_t=\begin{pmatrix} I_n&B(t)\\ 0_n&D(t)\end{pmatrix}$. Because it is symplectic, we have: $D(t)=I_n$ and $A_t^N=\begin{pmatrix} I_n&NB(t)\\ 0& I_n\end{pmatrix}$. This implies that all the Lyapunov exponents of the atomic measure supported at $\zeta $ are zero.

 \end{proof}

Let us now assume that $\mu$ is an invariant ergodic measure with at least one non-vanishing Lyapunov exponent the support of which is contained in a certain $C^0$ Lagrangian invariant graph $\Gamma$. By proposition \ref{P1}, the support of $\mu$ contains no critical points of $H$. As noticed in the previous remark \ref{previous}, we can choose a point $\zeta $ in the support of $\mu$ and two embedded disks $D^u$ and $D^s$ containing $\zeta $, the first one in the unstable set of $\zeta $, the second one in the stable set of $\zeta$. Then $D^u\cup D^s\subset \Gamma$.

\subsection{Proof of Theorem \ref{T1} in the   non-uniformly hyperbolic case}\label{ss21}
Let us assume that there are exactly $2(n-1)$ non vanishing exponents (there is always 2 vanishing Lyapunov exponents, one in the flow direction and the other  one in the energy direction). Let us define the weak local stable and unstable manifolds $\displaystyle{ W^s(\zeta )=\bigcup_{t\in(-\varepsilon, \varepsilon)}\phi^H_t(D^s)}$ and $\displaystyle{W^u(\zeta )=\bigcup_{t {\in} (-\varepsilon, \varepsilon)}\phi^H_t(D^u)}$. Then they are $n$-dimensional $C^{k-1}$ submanifolds that are contained in $\Gamma$ and such that $W^u(\zeta )\cap W^s(\zeta )=\{ \phi^H_t(\zeta ); x\in (-\varepsilon, \varepsilon)\}$. This cannot happen in the $n$-dimensional (topological) manifold $\Gamma$.\\

Hence in this case the proof just uses simple topological arguments. But we cannot use the same strategy when there are fewer non-zero Lyapunov exponents.

\subsection{Proof of Theorem \ref{T1} in the   general  case}\label{ss2}
We introduce the notation: $d=\dim(D^u)=\dim (D^s)$. We deduce from theorem 2 of \cite{arnaud2} that is recalled in the appendix too that at $\mu$ almost every point, we have: $\dim\big(G_-(\zeta )\cap G_+(\zeta )\big)= n-d$, hence we can assume this equality for the chosen $\zeta $.

Because $G_-(\zeta )$ and $G_+(\zeta )$ are $n$-dimensional and transverse to the vertical, their image by $D\pi(\zeta )$ is $T_xM$. Let us now fix a chart $U$ of $M$ at $x=\pi(\zeta )$. We use the notations: $g(x)=D\pi\big(G_-(\zeta )\cap G_+(\zeta )\big)$. Identifying $U$ with a part of $\R^n$ via the chart, we say that $g$ defines an affine $(n-d)$-dimensional foliation $\Gc$ of $U$.

 As $G_-(\zeta )=T_\zeta D^s\oplus \big(G_-(\zeta )\cap G_+(\zeta )\big)$ and $G_+(\zeta )=T_\zeta D^u\oplus \big(G_-(\zeta )\cap G_+(\zeta )\big)$ (see \cite{arnaud2}), the leaves of the foliation $\Gc$ are transverse to $d^u=\pi(D^u)$ and $d^s=\pi(D^s)$ if $U$ is small enough.
 
 Let us now choose a non-zero vector $v^s\in T_\zeta D^s\backslash \{0\}$ and let us use the notation: $w^s=D\pi(\zeta )\cdot v^s$. Because  $G_+(\zeta )=T_\zeta D^u\oplus \big(G_-(\zeta )\cap G_+(\zeta )\big)$, we have $T_xM=D\pi(\zeta )(T_\zeta D^u)\oplus g(x)$. Hence there exists $v^u\in T_\zeta D^u$ and $v\in g(x)$ such that $w^s=w^u+v$ if $w^u=D\pi(\zeta )\cdot v^u$. Using the definition of $T_xd^u$, $T_xd^s$ and the transversality of the foliation $\Gc$ to $d^u$ and $d^s$, we deduce the existence of $\varepsilon >0$ and for every $t\in [0,\varepsilon]$ of: $x^s(t)=x+tw^s+o(t)\in d^s$, $x^u(t)=x+tw^u+o(t)\in d^u$ such that $x^s(t)-x^u(t)\in g$.
 
 Then we use the notation: $\zeta ^u(t)=\pi_{|\Gamma}^{-1}\big(x^u(t)\big)$  and $\zeta ^s(t)=\pi_{|\Gamma}^{-1}\big(x^s(t)\big)$. Because $\Gamma$ is a Lipschitz graph containing $D^u$ and $D^s$, we deduce that we have  in chart: 
 $$\zeta ^u(t)=tv^u +o(t)\in D^u\quad {\rm and} \quad  \zeta ^s(t)=tv^s+o(t).$$
   Let us recall (see \cite{arnaud2}) that $G_-(\zeta )+G_+(\zeta )=\big(G_-(\zeta )\cap G_+(\zeta )\big)\oplus T_\zeta D^u \oplus T_\zeta  D^s$. As $v^s\not=0$, we have then: $v^u-v^s\notin G_-(\zeta )\cap G_+(\zeta )$. Hence (in chart): 
 $$(*)\lim_{t\rightarrow 0} \frac{1}{t}\big(\zeta ^u(t)-\zeta ^s(t)\big)=v^u-v^s\notin G_-(\zeta )\cap G_+(\zeta ).$$
 
 The submanifold $\Gamma$ is the graph of a Lipschitz map $\gamma: U\rightarrow \R^n$. By Rademacher theorem, $\gamma$ is Lebesgue almost everywhere differentiable. Let us assume that $\gamma$ is differentiable almost everywhere along the segment $[\pi\circ \zeta ^u(t), \pi\circ \zeta ^s(t)]$ for a while.
 Then  $\zeta ^u(t)-\zeta ^s(t)=\Big(x^u(t)-x^s(t)\, ,\,  \gamma\big(x^u(t)\big)-\gamma(x^s(t)\big)\Big)$ is equal to: 
 $$(**)\quad \zeta ^u(t)-\zeta ^s(t)=\int_0^1 \Big(x^u(t)-x^s(t)\, ,\,  D\gamma\Big(x^s(t)+\sigma(x^u(t)-x^s(t)\big)\Big).\big(x^u(t)-x^s(t)\big)\Big)d\sigma.$$
 Proposition 4.3. of \cite{arnaud1} states that:
 $$\begin{matrix} G_-\circ\pi_{|\Gamma}^{-1}\Big(x^s(t)+\sigma\big(x^u(t)-x^s(t)\big)\Big)&\leqslant T_{\pi_{|\Gamma}^{-1}\big(\textrm{$x^s(t)+\sigma\big(x^u(t)-x^s(t)\big)$}\big)}\Gamma\hfill\\
& \leqslant G_+\circ\pi_{|\Gamma}^{-1}\Big(x^s(t)+\sigma\big(x^u(t)-x^s(t)\big)\Big),\end{matrix}$$
 i.e. that the tangent Lagrangian subspace to $\Gamma$ is between the two Green bundles at the points where $\gamma$ is differentiable. Moreover, $G_-\leqslant G_+$, $G_-$ is lower semicontinuous and $G_+$ is upper semicontinuous (see \cite{arnaud1}). 
 
 We introduce the following notations: $G_\pm(\eta)$ is the graph of the symmetric matrix $S_\pm(\eta)$. Because of the semicontinuity of $G_-$ and $G_+$, there exist $\Delta S_\pm(\eta)$ a semi-positive   matrix that depends on $\eta$, vanishes at $\zeta $ and is continuous at $\zeta $ such that:
 $S_+(\eta)\leqslant S_+(\zeta )+\Delta S_+(\eta)$ and $S_-(\zeta )-\Delta S_-(\eta)\leqslant S_-(\eta)$. We obtain then at every point $\eta$ where $\Gamma$ is differentiable:
 $$S_-(\zeta )-\Delta S_-(\eta)\leqslant S_-(\eta)\leqslant D\gamma \big(\pi(\eta)\big)\leqslant S_+(\eta)\leqslant S_+(\zeta )+\Delta S_+(\eta).$$
 Let us consider the restrictions (as quadratic forms) of the previous matrices to $g(x)$ and let us denote them with a ``$\widetilde{\quad}$''. We have then    $\widetilde S_+(\zeta )=\widetilde S_-(\zeta )$ and: 
$$\widetilde S_-(\zeta )-\Delta \widetilde S_-(\eta)\leqslant \widetilde S_-(\eta)\leqslant \widetilde {D\gamma \big(\pi(\eta)\big)}\leqslant\widetilde S_+(\eta)\leqslant \widetilde S_+(\zeta )+\Delta \widetilde S_+(\eta).$$
Moreover, $\displaystyle{\lim_{\eta\rightarrow \zeta }\Delta\widetilde S_\pm(\eta)=0}$. This implies that:
$$(***)\hfill\lim_{x'\rightarrow x} \widetilde{D\gamma(x')}=\widetilde S_\pm(\zeta ).$$
Let us prove that this implies that 
$$(****)\quad\displaystyle{\lim_{\eta\rightarrow \zeta } D\gamma\big(\pi(\eta)\big)_{|g}=S_\pm(\zeta )_{|g}}.$$
 The map $\gamma$ being Lipschitz, the $D\gamma(x')$ are uniformly bounded. Hence, if $(****)$ is not true, we can find a sequence $(x_k)$ converging to $x$ such that $\big(D\gamma(x_k)\big)_{k\in \N}$ converges and $\displaystyle{\lim_{k\rightarrow \infty}   {D\gamma(x_k)}_{|g}\not=\  S_+(\zeta )_{|g}}.$  As $D\gamma(x_k)\leqslant S_+(x_k)$ and as $S_+$ is upper semicontinuous, we know that $\displaystyle{   S_+(\zeta )-\lim_{k\rightarrow \infty}   D\gamma(x_k)}$ is positive semi-definite, hence its kernel coincides with its isotropic cone. And we have proved in $(***)$ that $g$ is in this isotropic cone. This implies $\displaystyle{\lim_{k\rightarrow \infty}   {D\gamma(x_k)}_{|g} =\  S_+(\zeta )_{|g}}$ and gives a contradiction. 
 
 Hence if $\eta$ is close to $\zeta $ and such that $\gamma$ is differentiable at $\eta$: $D\gamma\big(\pi(\eta)\big)_{|g}=S+\alpha (\eta)$ where $\displaystyle{\lim_{\eta\rightarrow \zeta }\alpha(\eta)=0}$ and $G_-(\zeta )\cap G_+(\zeta )$ is the graph of $S$ above $g$. Replacing in $(**)$, we obtain: $\zeta ^u(t)-\zeta ^s(t)=$
 $$\Big( x^u(t)-x^s(t)\, ,\, \int_0^1\Big[S+\alpha\Big(x^s(t)+\sigma\big(x^u(t)-x^s(t)\big)\Big)\Big]d\sigma \cdot \big(x^u(t)-x^s(t)\big)\Big)$$
 i.e:
 $$\zeta ^u(t)-\zeta ^s(t)= \big(t(w^u-w^s)+o(t)\, ,\,  tS(w^u-w^s)+o(t)\big)$$
  and we deduce:
 $$\lim_{t\rightarrow 0} \frac{1}{t}\big(\zeta ^u(t)-\zeta ^s(t)\big)=\big(w^u-w^s\, ,\,  S(w^u-w^s)\big)\in G_-(\zeta )\cap G_+(\zeta ),$$
 which contradicts $(*)$.
 
 If $\gamma$ is not differentiable almost everywhere along the segment $[\zeta ^u(t), \zeta ^s(t)]$, using Fubini theorem, we can replace $\zeta ^u(t)$ and $\zeta ^s(t)$ by $\eta^u(t)=\zeta ^s(t)+o(t)$ and $\eta^s(t)=\zeta ^s(t)+o(t)$ in such a way that $\gamma$ is  differentiable almost everywhere along the segment $[\eta^u(t), \eta^s(t)]$, and we use the same argument as before to conclude.
 
 \section*{Appendix}
\subsubsection*{Aubry sets}
 If $\lambda$ is a $C^\infty$ closed 1-form of $\T^n$, then the map
$T_\lambda~: T^*\T^n\rightarrow T^*\T^n$ defined by~: $T_\lambda (q, p)=(q,
p+\lambda (q))$ is a symplectic  $C^\infty$  diffeomorphism; therefore,
we have~: $(\phi^{H\circ T_\lambda}_t)=(T_\lambda^{-1}\circ \phi_t\circ
T_\lambda )$, i.e.  the Hamiltonian flow of $H$ and $H\circ T_\lambda$ are
conjugated. Moreover, the Tonelli Hamiltonian function $H\circ T_\lambda$ is
associated to the Tonelli Lagrangian function $L-\lambda$, and it is
well-known that~: $(\varphi_t^L)=(\varphi_t^{L-\lambda})$; the two Euler-Lagrange
flows are equal. Let us emphasize that these flows are equal, but the
Lagrangian functions, and then the Lagrangian actions differ.

For   a Tonelli Lagrangian function ($L$ or $L-\lambda$), J.~Mather introduced in \cite{Ma1}
(see  \cite{M2} too) a particular subset $\Ac(L-\lambda)$ of $T\T^n$
which he called the ``static set'' and which is now usually called the
``{\em Aubry set}''. There exist  different but equivalent definitions of
this set (see
\cite{CIPP} ,  
\cite{Fathi},
\cite{M2}) and it is known that two closed 1-forms
that are in  the same cohomological class define the same Aubry set:
$$[\lambda_1]=[\lambda_2]\in H^1(\T^n,\R)\Rightarrow \Ac (L-\lambda_1)=\Ac
(L-\lambda_2).$$
We can then   introduce the following notation: if $c\in H^1(\T^n,\R)$ is
a cohomological class, we have
$\Ac_c=\Ac_c(L)=\Ac (L-\lambda)$  where $\lambda$ is any closed 1-form belonging
to $c$. Then $\Ac_c$ is compact, non empty and invariant under $(\varphi_t^L)$.
Moreover, J.~Mather proved in \cite{Ma1} that it is  a Lipschitz graph  above a part of the
zero-section (see
\cite{Fathi}   too). 

As we are  interested in the Hamiltonian dynamics as well as in the
Lagrangian ones, let us  define the dual Aubry set:

\begin{enumerate}
\item[--] if $H$ is the Hamiltonian function associated to the Tonelli
Lagrangian function $L$, its {\em dual Aubry set} is $\Ac^*(H)=\Lc
\big(\Ac (L)\big)$; 
\item[--] if $c\in H^1(\T^n,\R)$ is a cohomological class, then
$\Ac^*_c=\Ac^*_c(H)=\Lc\big(\Ac_c(L)\big)$ is the {\em $c$-dual Aubry set}; let us
notice that for any closed 1-form
$\lambda$ belonging to $c$, we have: $T_\lambda(
\Ac^*\big(H\circ T_\lambda)\big)=\Ac_c^*(H)$.
\end{enumerate}
These sets are invariant by the Hamiltonian flow $(\phi_t^H)$.

Then  there
exists a real number denoted by
$\alpha_H (c)$ such that~: $\Ac^*_c\subset H^{-1} \big(\alpha_H (c)\big)$ (see
\cite{Car} and \cite{Ma}), i.e. each dual Aubry set is contained in an energy level.
\\
The following property is a well-known characterization of the projected Aubry set: \\
  $x_0\in\T^n$ is such that there exists a sequence of absolutely continuous curves $\gamma_k:[0, T_k]\rightarrow \T^n$, with $(T_k) \to \infty$,  such that $\gamma_k(0)=\gamma_k(T_k)=x_0$ and $$\displaystyle{\lim_{k\rightarrow +\infty}\int_0^{T_k}\big(L(\gamma_k, \gamma'_k)-\lambda(\gamma'_k)+\alpha_H(c)\big)=0},$$ 
  if and only if  $x_0\in \pi(\Ac_c)$.

The following proposition is proved in \cite{arnaud4}:
\begin{propo}\label{Pradial}  Let $c\in H^1(\T^n, \R)$ and $\lambda\in c$, $\varepsilon >0$  and let
$L: T\T^n\rightarrow
\R
$ be a Tonelli Lagrangian function. Then there exists $T_0>0$ such
that: \\
$\forall T\geqslant T_0, \forall (x_0,v_0)\in \Ac_c, \forall \gamma: [0, T]\rightarrow \T^n$ minimizing for $L-\lambda$ between $x_0$ and $x_0$, i.e.:\\
 \begin{multline*}
 \forall \eta: [0, T]\rightarrow \T^n, \\\quad \eta (0)=\eta (T)=x_0
 \Rightarrow  \int_0^T\big(L(\gamma, \gamma')-\lambda (\gamma')+\alpha_H(c)\big)\leqslant  \int_0^T\big(L(\eta, \eta')-\lambda (\eta')+\alpha_H(c)\big)
\end{multline*}
 \\
 then we have: $d\big((x_0, v_0), (x_0, \gamma'(0))\big)\leqslant \varepsilon$
\end{propo}

\subsubsection*{Mather sets}

The general references for this section are \cite{Ma} and \cite{Masor}.
Let ${\cal M} (L)$ be the space of compactly supported Borel probability measures that are invariant by the Euler-Lagrange flow $(\varphi_t^L)$.  To every $\mu\in {\cal M} (L)$ we  associate its average action $A_L(\mu)=\int_{T\T^n}Ld\mu$. It is proved in \cite{Ma} that for every $f\in C^1(\T^n, \R)$, we have: 
$$\int df(q).v d\mu(q,v)=0.$$Therefore we can define on $H^1(\T^n, \R)$ a linear functional $\ell(\mu)$ by: $$\ell(\mu)([\lambda])=\int\lambda (q)\cdot vd\mu (q,v)$$ (here $\lambda$ designates any closed 1-form). Then  there exists a unique element $\rho(\mu)\in H_1(\T^n, \R)$ such that:
$$\forall \lambda,\quad \int_{T\T^n}\lambda (q)\cdot vd\mu(q,v)=[\lambda]\cdot \rho(\mu).$$
The homology class  $\rho(\mu)$ is called the {\em rotation vector} of $\mu$. Then the map $\mu\in{\cal M} (L)\rightarrow \rho(\mu)\in H_1(\T^n, \R)$ is   onto.   Mather $\beta$-function $\beta: H_1(\T^n, \R)\rightarrow \R$  associates to each homology class $h\in H_1(\T^n, \R)$ the minimal value of the average action $A_L$ over the set of measures of ${\cal M}(L)$ with rotation vector $h$. We have:
$$\beta (h)=\min_{\substack{\mu\in{\cal M} (L)\\ \rho(\mu)=h}}A_L(\mu).$$
A measure $\mu\in {\cal M} (L)$ realizing such a minimum, i.e. such that $A_L(\mu)=\beta \big(\rho(\mu)\big)$ is called a {\sl minimizing measure with rotation vector} $\rho (\mu)$. The $\beta$ function is convex and superlinear,  and  its conjugate function (given by Fenchel duality) $\alpha: H^1(\T^n, \R)\rightarrow \R$ is defined  by:
$$\alpha ([\lambda]) =\max_{h\in H_1(\T^n, \R)}\big([\lambda]\cdot h-\beta (h)\big)=-\min_{\mu\in{\cal M}(L)}A_{L-\lambda}(\mu).$$
A measure $\mu\in{\cal M} (L)$ realizing the minimum of $A_{L-\lambda }$ is called a {\sl $[\lambda]$-minimizing measure}. Observe that the function $\alpha$ is exactly the same as the function $\alpha_H$  defined in the section on Aubry sets. It is convex and superlinear.\\
Being convex, Mather's $\beta$ function has a subderivative at any point $h\in H_1(\T^n, \R)$; i.e. there exists $c\in H^1(\T^n, \R)$ such that:
 $$\forall k\in H_1(\T^n, \R),\quad  \beta (h)+c\cdot (k-h)\leqslant \beta (k).$$
  We denote by $\partial \beta (h)$ the set of all the subderivatives of $\beta$ at $h$. By Fenchel duality, we have:
$c\in\partial\beta (h)\Leftrightarrow c\cdot h=\alpha (c)+\beta (h)$.\\
Then we introduce the following notations:
\begin{enumerate}
\item[$\bullet$] if $h\in H_1(\T^n, \R)$, the Mather set for the rotation vector $h$ is: 
$$\displaystyle{\Mc^h(L)=\bigcup\{ {\rm supp}(\mu);\quad  \textrm{$\mu$ is minimizing with rotation vector $h$}\}};$$
\item[$\bullet$] if $c\in H^1(\T^n , \R)$, the Mather set for the cohomology class $c$ is: 
$$\displaystyle{\Mc_c(L)=\bigcup\{ {\rm supp}(\mu);\quad  \textrm{$\mu$ is $c$-minimizing}\}};$$
\end{enumerate}
where ${\rm supp}(\mu)$ designates the support of the measure $\mu$.\\
The sets $\Mc^h(L)$ and $\Mc_c(L)$ are invariant by $\varphi_t^L$.\\
The following equivalences are proved in \cite{Masor} for any pair $(h, c)\in H_1(M, \R)\times H^1(M, \R)$: 
$$\Mc^h(L)\cap \Mc_c(L)\not=\varnothing\quad \Longleftrightarrow \quad \Mc^h(L)\subset \Mc_c(L)\quad \Longleftrightarrow \quad c\in\partial\beta (h).$$
The dual Mather set for the cohomology class $c$ is defined by: 
$\Mc_c^*(H)=\Lc\big(\Mc_c(L)\big)$. If $\cal M^*(H)$ designates the set of compactly supported Borel probability measures of $T^*M$ that are invariant by the Hamiltonian flow $(\phi^H_t)$, then the map $\Lc_*: {\cal M}(L)\rightarrow {\cal M}^*(H)$ that push forward the measures by $\Lc$ is a bijection. We denote $\Lc_*(\mu)$ by  $\mu^*$ and say that the measures are dual. We say too that $\mu^*$ is minimizing if $\mu$ is minimizing in the previous sense. \\
Moreover, the Mather set $\Mc_c^*(H)$ is a subset of the Aubry set $\Ac_c^*(H)$ and every invariant Borel probability measure the support of whose is in $\Ac_c^*(H)$ is $c$-minimizing.

\subsubsection*{Ma\~n\'e sets}
The Ma\~n\'e set  $\Nc (L)$ of $L$ is the set  of $\big(\gamma(0), \gamma'(0)\big)\in T\T^n$ such that for all segment $[a, b]\subset \R$,  $\gamma_{|[a, b]}$ is a minimizer for $L$. The {\em dual
Ma\~n\'e set} is then
$\Nc^*(H)=\Lc\big(\Nc(L)\big)$.\\
 For all $c\in H^1(\T^n, \R)$ and $\lambda\in
c$, then $\Nc_c=\Nc_c(L)=\Nc(L-\lambda)$ is independent  of the choice of
$\lambda\in c$ and  the {\em $c$-dual Ma\~n\'e set} is
$\Nc_c^*(H)=\Lc\big(\Nc_c(L)\big)=T_\lambda\big(\Nc^*(H\circ T_\lambda)\big)$. It  is
invariant under $(\phi_t^H)$, compact and non empty but is not
necessarily a graph.

For every cohomological class $c\in H^1(\T^n)$, we have the inclusion~:
$\Mc^*_c(H)\subset \Ac^*_c(H)\subset  \Nc^*_c(H)  \subset H^{-1} \big(\alpha_H (c)\big)$ (see
\cite{Car} and \cite{Ma}), i.e. each dual Ma\~n\'e set is contained in an energy level.\\
Moreover, the $\omega$ and $\alpha$-limit sets of every point of the Ma\~n\'e set $\Nc^*_c(H) $ are
contained in the Aubry set $\Ac^*_c(H)$.

\subsubsection*{The link with the weak KAM theory}
The reference for this section is \cite{Fathi}. We just recall some results that are used in the article;   a {\em $C^0$ Lagrangian} graph is the graph of $a+du:\T^n\rightarrow (\R^n)^*$ where $a\in (\R^n)^*$ and $u\in C^1(\T^n, \R)$. Then $a\in H^1(\T^n, \R)$ is the {\em cohomology class} of the graph. We have:\\
if $\Gc$ is a  Lagrangian graph with cohomology class $c$ that is invariant by $\Phi_t$, then $$\Ac^*_c\subset \Gc\subset \Nc^*_c.$$
Moreover, if $\Ac^*_c$ (resp. $\Nc^*_c$) is  a graph above the whole $\T^n$, then we have $\Ac^*_c=\Nc^*_c$ and it is a $C^0$ Lagrangian graph.\\
It is proved in \cite{Fathi2} that every $C^0$ Lagrangian graph that is invariant by a Tonelli Hamiltonian is a Lipschitz graph.

\subsubsection*{Green bundles}

Recall 
(see \cite{Cont} and \cite{arnaud1} for details) that if $s \in \R \longmapsto (x,p) = 
\phi_s^H(x_0,p_0) \in T^*\T^n$ is an orbit of the Hamiltonian flow that is free of conjugate points, one may define 
two  bundles $G_-$ and $G_+$ (called the    Green bundles) by 
$$G_+(x,p) = \lim _{t \longrightarrow +\infty} D \phi_t^H \big( \phi_{-t}^H(x,p) \big) \cdot V^*\big( \phi^H_{-t}(x,p) \big)$$
and  $$G_-(x,p) = \lim _{t \longrightarrow +\infty} D \phi_{-t}^H \big( \phi_{t}^H(x,p) \big) \cdot V^*\big( \phi^H_{t}(x,p) \big).$$
Then $G_-$ is the negative Green bundle and $G_+$ is the positive one.

\noindent
Every $G_\pm(x,p)$ is a Lagrangian subspace of $T_{(x,p)}T^*\T^n$ that is transverse to the vertical space $V^*(x,p)$, and this bundle is invariant by the Hamiltonian flow: 
$D \phi_t^H G_\pm(x,p) = G_\pm\big( \phi^H_t(x,p) \big)$ for all $t \in \R$. 

We have of the following criteria (see \cite{Cont} and \cite{arnaud1}): if $w \in T_{(x,p)}(T^*\T^n)$, 
then $$ w \notin G_+(x,p) \ \Longrightarrow \ \lim_{t \rightarrow + \infty} 
\vv D (\pi \circ \phi^H_{-t})(x,p) \cdot w \vv = + \infty,$$
$$ w \notin G_-(x,p) \ \Longrightarrow \ \lim_{t \rightarrow + \infty} 
\vv D (\pi \circ \phi^H_{t})(x,p) \cdot w \vv = + \infty,$$
\noindent
where $\vv \cdot \vv$ denotes the Euclidean norm.\\
Moreover, $G_+$ is upper semi-continuous and $G_-$ is lower semi-continuous, and we have at every point: $G_-\leqslant G_+$ (for the usual order relation on the Lagrangian subspaces that are transverse to the vertical, given by the order on symmetric matrices, see \cite{arnaud1} for details). Hence $\{ G_-=G_+\}$ is a $G_\delta$ subset of $T^*\T^n$.\\
It is proved in \cite{arnaud1} that if $G$ is any invariant Lagrangian subspace that is transverse to the vertical space (for example the tangent to some invariant Lipschitz Lagrangian graph), then we have: $G_-\leqslant G\leqslant G_+$.

There is a strong link between Oseledet's bundle and Green bundles, as explained in \cite{arnaud2}:
\begin{theo}
Let $H~: T^*\T^n\rightarrow \R$ be a Tonelli Hamiltonian and let $\mu$ be an ergodic minimizing probability measure. Then the two following assertions are equivalent:
\begin{enumerate}
\item[$\bullet$] at $\mu$ almost every point, $\dim \big(G_-(x)\cap G_+(x)\big)=p$;
\item[$\bullet$] $\mu$ has exactly $2p$ zero Lyapunov exponents, $n-p$ positive ones and $n-p$ negative ones.
\end{enumerate}
\end{theo}
Moreover, if the Oseledet's splitting along the support of $\mu$ is denoted by $E^s\oplus E^c\oplus E^u$, then we have: $G_-=E^s\oplus(G_-\cap G_+)$ and $G_+=E^u\oplus (G_-\cap G_+)$.
 

\bibliography{synthese}
\bibliographystyle{alpha}

\end{document}